\UseRawInputEncoding
\documentclass[11pt]{article}

\usepackage{calrsfs,amsmath,mathrsfs,amsthm,graphicx,amsfonts,color,stmaryrd,bbm,bm,authblk} 
\newcommand*\samethanks[1][\value{footnote}]{\footnotemark[#1]}

\newcommand{\R}{{\mat R}}

\newcommand{\beq}{\begin{equation}}
\newcommand{\eeq}{\end{equation}}

\newcommand{\be}{\begin{eqnarray}}
\newcommand{\ben}{\begin{eqnarray*}}
\newcommand{\en}{\end{eqnarray}}
\newcommand{\enn}{\end{eqnarray*}}

\newcommand{\pa}{\partial}

\newcommand{\ov}{\overline}

\newcommand{\Om}{\Omega}

\newcommand{\wh}{\widehat}

\newcommand{\mat}{\mathbb}

\newcommand{\bv}{{\bf{v}}}
\newcommand{\bu}{{\bf{u}}}
\newcommand{\bts}{\text{\bf{\sf{t}}}}
\newcommand{\bn}{{\bf{n}}}
\newcommand{\bw}{{\bf{w}}}
\newcommand{\bp}{{\bf{p}}}
\newcommand{\bzero}{\boldsymbol{0}}
\newcommand{\bt}{{\bf{t}}}
\newcommand{\bC}{{\bf{C}}}
\newcommand{\bPhi}{\boldsymbol{\Phi}}
\newcommand{\bpsi}{\boldsymbol{\psi}}
\newcommand{\bV}{{\bf{V}}}

\newcommand{\bx} {\boldsymbol{x}}
\newcommand{\by} {\boldsymbol{y}}
\newcommand{\norms}[1]{\parallel\! #1 \!\parallel} 
\newcommand{\bz} {\boldsymbol{z}}
\newcommand{\bg} {{\boldsymbol{g}}}
\newcommand{\bzeta} {{\boldsymbol{\zeta}}}
\newcommand{\bI} {{\bf I}}

\newcommand{\bphi} {\boldsymbol{\phi}}

\newcommand\exs{\hspace*{0.4mm}}
\newcommand\xxs{\hspace*{0.2mm}}

\newcommand\nxs{\hspace*{-0.2mm}}

\newcommand{\lb}{\label}

\definecolor{rot}{rgb}{0,0,0}
\definecolor{nn}{rgb}{0,0,0}
\definecolor{hw}{rgb}{0,0,0}
\definecolor{xl}{rgb}{0,0,0}
\definecolor{nnn}{rgb}{0,0,0}
\definecolor{hw1}{rgb}{0,0,0}
\definecolor{lxl}{rgb}{0,0,0}
\newtheorem{theorem}{Theorem}[section]
\newtheorem{assum}{Assumption}[section]
\newtheorem{lemma}[theorem]{Lemma}

\newtheorem{remark}[theorem]{Remark}

\newtheorem{proposition}[theorem]{Proposition}

\begin{document}

\title{\bf Time- vs. frequency- domain inverse elastic scattering: Theory and experiment}

\author[1]{X.~Liu\thanks{XL and JS contributed equally to this work.}}
\author[2]{J.~Song\samethanks}
\author[2,3]{F.~Pourahmadian\thanks{Corresponding author email: fatemeh.pourahmadian@colorado.edu}}
\author[4]{H.~Haddar}
\affil[1]{School of Mathematical Sciences, Beihang University, Beijing, 100191, CHINA}
\affil[2]{Department of Civil, Environmental \& Architectural Engineering, University of Colorado Boulder, USA}
\affil[3]{Department of Applied Mathematics, University of Colorado Boulder, USA}
\affil[4]{INRIA, Center of Saclay Ile de France and UMA, ENSTA Paris Tech, Palaiseau Cedex, FRANCE}
\date{} 
\renewcommand\Affilfont{\itshape\small}

\maketitle

\vspace{-7mm}
\begin{abstract}
This study formally adapts the time-domain linear sampling method (TLSM) for ultrasonic imaging of stationary and evolving fractures in safety-critical components. The TLSM indicator is then applied to the laboratory test data of~\cite{Yue2021,pour2021} and the obtained reconstructions are compared to their frequency-domain counterparts. The results highlight the unique capability of the time-domain imaging functional for high-fidelity tracking of evolving damage, and its relative robustness to sparse and reduced-aperture data at moderate noise levels. A comparative analysis of the TLSM images against the multifrequency LSM maps of~\cite{Yue2021} further reveals that thanks to the full-waveform inversion in \emph{time} and space, the TLSM generates images of remarkably higher quality with the same dataset.   
\end{abstract}


\section{Introduction}

Recent laboratory implementations~\cite{baro2018,Yue2021} of the linear sampling method (LSM)~\cite{Col1996,Fiora2008} for ultrasonic imaging showcase a unique opportunity for almost real-time reconstruction of anomalies with exceptional resolution and flexibility in terms of sensing configuration.  In~\cite{baro2018,Yue2021}, the data inversion is conducted in the frequency domain by deploying the most pronounced spectral components of the (time-domain) measurements. More specifically,~\cite{baro2018} uses an adaptation of LSM in the modal space to recover the support of damage in an elastic waveguide, while~\cite{Yue2021} directly computes the sampling indicator from the Fourier-transformed boundary measurements to reconstruct a partially-closed stationary fracture in an elastic plate. Demonstrating success in imaging with dense datasets, these studies simultaneously expose the sensitivity of the frequency-domain LSM to noise especially with sparse data. The latter was displayed by the emergence of many reconstruction artifacts and failure to recover parts of the hidden scatterer. To resolve this,~\cite{Yue2021} applies the generalized linear sampling method (GLSM)~\cite{Audibert2014,Fatemeh2017} to the same dataset and the results show remarkable improvement. The GLSM furnishes a more robust imaging tool by eliminating a heuristic assumption involved in the design of LSM imaging functional. More specifically, the GLSM takes advantage of a symmetric factorization of the scattering operator along with coercivity of the resulting middle operator to carefully construct a new cost function whose minimizer carries suitable properties for a more stable imaging indicator. The robustness of GLSM, however, comes with the cost of a slower reconstruction due to the more complex minimization of its associated cost function. Also, while the more rigorously built GLSM generates higher quality images, the inversion still occurs in the frequency domain which could be another source of sensitivity to sparse imaging -- since instead of full-length time signals, only a discrete subset of their spectra is used for computing the LSM indicator maps. In light of this, there are ongoing efforts to formally extend the GLSM indicator for inverse scattering in the time domain, e.g.,~\cite{cako2019,Liu2021}.              

Motivated by the promise of spatiotemporal full-waveform inversion and in light of recent developments on inverse electromagnetic and acoustic scattering in the time domain~\cite{chen2010,HLM14,guo2016, Selgas2021}, this study rigorously formulates the time-domain LSM for elastic-wave imaging of crack networks in solids. The TLSM indicator is then applied to the laboratory test data in~\cite{Yue2021} and the results are compared to multifrequency LSM reconstructions from the same dataset. It should be mentioned that~\cite{Yue2021} is focused on single-step imaging of \emph{stationary} scatterers where experiments are conducted on (a) intact specimen before mounting in a MTS load frame for fracturing, and (b) fractured specimen after dismounting at $60\%$ of the maximum load in the post peak regime. The ultrasonic measurements in (a) and (b) are then used to compute the scattering signatures of a partially closed fracture in the specimen for constructing the LSM maps. On the other hand,~\cite{pour2021} reports a complementary suit of ultrasonic experiments conducted \emph{during fracturing} when the same specimen is in the load frame. These measurements have so far been only used for sequential recovery of (geometric and interfacial) \emph{evolution} via the differential imaging method~\cite{pour2019}. In this study, we take advantage of the dataset in~\cite{pour2021} at $75\%$ and $90\%$ of the maximum load for single-step reconstruction via TLSM to further examine the capacity of this indicator for tracking of evolving anomalies.         

This paper is organized as follows. {Section}~\ref{PS} presents the direct scattering problem and the affiliated dataset for inversion. Relevant function spaces along with the admissibility conditions for parameters, such that the forward problem remains wellposed, are discussed in {Section}~\ref{Wellp}. {Section}~\ref{NfO} defines the near-field elastic scattering operator and its factorization. This is followed by establishing some results on the properties of involved operators. Based on the latter, the time-domain LSM indicator is introduced in {Section}~\ref{INS}. {Section}~\ref{exp_set} is dedicated to implementation of this imaging modality to laboratory test data and comparing the results with the corresponding frequency-domain reconstructions. Finally, a summary of the main findings is provided in {Section}~\ref{Conc}.

\section{Problem statement}\lb{PS}
\setcounter{equation}{0}

We consider the elastic-wave sensing of a fracture $\Gamma\subset\R^3$ embedded in a homogeneous, isotropic, elastic solid endowed with the mass density $\rho$ and Lam\'{e} parameters $\mu$ and $\lambda$. The fracture is characterized by a heterogeneous contact condition to describe the spatially-varying nature of its rough  interface. For a given vector $\bp\in\mathbb{R}^3$, $\Gamma$ is illuminated by an incident point source which is convolution in time of the Green dyadic with a generic pulse $\chi$,
\be\label{eqinc}
\bu_\chi^i(\bx,t;\by,\bp) := \left[\chi(\cdot)*\boldsymbol{\Pi}(\bx,\cdot;\by)\bp\right](t),\quad (\bx,t)\in \R^3\backslash\{\by\}\times\R.
\en
Here, $\boldsymbol{\Pi}$ is the fundamental displacement tensor which may be recast as
\ben
\begin{aligned}
\boldsymbol{\Pi}(\bx,t;\by)~=~&\frac{1}{\mu}\left(\bI_2-\nabla_{\bx}\otimes\nabla_{\bx}\right) 
\frac{\delta\left(t-|\bx-\by|/\sqrt{\mu/\rho}\right)}{|\bx-\by|} ~+\,  \\*[0.2mm]
&\frac{1}{\lambda+2\mu}\nabla_{\bx}\otimes\nabla_{\bx}\frac{\delta\left(t-|\bx-\by|/\sqrt{(\lambda+2\mu)/\rho}\right)}{|\bx-\by|}, \quad (\bx,t)\in \R^3\backslash\{\by\}\times\R,
\end{aligned}
\enn
where $\bI_2$ is the $3\!\times\!3$ identity dyadic. The corresponding scattered field $\bv$ solves
\be\label{eq1}
\left\{
\begin{array}{ll}
\nabla\cdot\left(\bC:\nxs\nabla \bv(\bx,t)\right)\,-\,\rho \ddot{\bv} (\bx,t) \,=~ \bzero \quad&\text{in}\quad \mathbb{R}^3\backslash\Gamma\times\R,\\[1.5mm]
{\bf n}\cdot\left(\bC:\nxs\nabla \bv(\bx,t)\right) \,=~ {\bf K}(\bx)[\![\bv(\bx,t)]\!]\,-\,{\bf t}^i(\bx,t) \quad&\text{on}\quad\Gamma\times\R,
\end{array}
\right.
\en
subject to the causality condition $\bv(\bx,t)=\bzero$ for $t<0$. The elasticity tensor $\bC$ is given by
\ben
\bC ~=~ \lambda \xxs {\bf I}_2 \nxs\otimes\nxs {\bf I}_2 \,+\, 2\mu{\bf I}_4,
\enn
with ${\bf I}_4$ denoting the $4$th-order symmetric identity tensor; $[\![\bv]\!]=[\bv^+-\bv^-]$ is the jump in $\bv$ across $\Gamma$; {${\bf t}^i = {\bf n}\cdot\left(\bC:\nabla {\bf u}_\chi^i \right)$} is the free-field traction vector; $\bf n = {\bf n}^-$ is the unit normal on $\Gamma$; {${\bf K=\bf K}(\bx)$} is a {\em symmetric} matrix of the specific stiffness coefficients.
\begin{remark}
	In what follows, all quantities are rendered dimensionless by
	taking $\rho, \mu$, and $R$--the characteristic size of a region sampled for fractures--as the respective scales for mass density, elastic modulus, and length, which amounts to setting $\rho = \mu = R = 1$~\cite{Scaling2003}.
\end{remark}
The {\em inverse problem} is to reconstruct $\Gamma$ from the partial knowledge of the scattered waves on some measurement surface $\Gamma_m\subset\R^3\backslash\Gamma$. The measured data set is
\ben
	\left\{\bv(\bx,t;\by,\bp) : \bx\in\Gamma_m, \, \by\in \Gamma_i, \, t\in\R, \,\bp\in \{{\bf{e}}_k\}_{k=1,2,3}\right\},
\enn
where $\bv(\bx,t;\by,\bp)$ is the scattered field for an incident point source emitted at $\by\in\Gamma_i\subset\R^3\backslash\Gamma$ and $\{{\bf{e}}_k\}_{k=1,2,3}$ is the unit coordinate vectors in $\R^3$.

\section{Well-posedness of the forward scattering problem}\lb{Wellp}
We shall analyze the scattering problem~\eqref{eq1} with $\rho=1$ by deploying the Laplace transform as in \cite{HLM14}. Given the Hilbert space $X$, we denote by $\mathcal{D}(\mathbb{R};X)=C^{\infty}_0(\mathbb{R};X)$ smooth and compactly supported $X$-value functions. Further, $\mathcal{D}'(\mathbb{R};X)$ are $X$-valued distributions on the real line and the corresponding tempered distributions are $\mathcal{S}'(\mathbb{R};X)$. For $\sigma\in\mathbb{R}$ we set
\ben
\mathcal{L}'_{\sigma}(\mathbb{R};X)~=~\left\{f\in\mathcal{D}'(\mathbb{R};X)\! : e^{-\sigma t}f(t)\in
\mathcal{S}'(\mathbb{R};X)\right\}.
\enn
For $f\in\mathcal{L}'_{\sigma}(\mathbb{R};X)$, define the Laplace transform with respect to the time variable as
\ben
(\mathscr{L}f)(s)~=~\widehat{f}(s)~=\,\exs\int^{\infty}_{-\infty}e^{is t}f(t)dt~=~\mathscr{F}(e^{-\sigma t} f)(\omega),\quad s~=~\eta+i\sigma.
\enn
For $\sigma\in\R$, we denote
\ben
\mathbb{C}_\sigma:=\left\{s\in\mathbb{C} : \Im s\leq\sigma\right\}.
\enn
Formally applying the Laplace transform to~\eqref{eq1}, observe that $\widehat{\bv}(\bx,s)$ satisfies
\be\label{eq2}
\left\{
\begin{array}{ll}
	\nabla\cdot\left(\bC:\nxs\nabla \widehat{\bv}(\bx,s)\right) \,+\, s^2 \xxs \widehat{\bv} (\bx,s) \,=~ \bzero \quad&\text{in}\quad \mathbb{R}^3\backslash\Gamma,\\[1.5mm]
{\bf n}\cdot\left(\bC:\nabla \widehat{\bv}(\bx,s)\right) \,-\, {\bf K}(\bx)[\![\widehat{\bv}(\bx,s)]\!]~=\, -\xxs\widehat{\bf t}^{i}(\bx,s)\quad&\text{on}\quad\Gamma.
\end{array}
\right.
\en

Now, the objective is to establish explicit bounds on $\widehat{\bv}(\bx,s)$ in terms of $s\in\mathbb{C}_{\sigma_0}$. In this vein, with reference to~\cite{H03}, let us define a frequency dependent norm on $H^1(D)^3$ for a Lipschitz domain $D$ as the following
\ben
\|\widehat{\bu}(\cdot,s)\|_{H^1_s(D)^3}:=\sqrt{\int_D\left(|\nabla\widehat{\bu}(\bx,s)|^2+|s\widehat{\bu}(\bx,s)|^2\right)d\bx},
\enn
which is equivalent to the usual norm $H^1(D)^3$ if $s\neq 0$. Similarly, frequency dependent norms exist for the trace spaces $H^{\pm 1/2}(\pa D)^3$ on the boundary, see~\cite{McLean2000} for the general definition. The latter may be defined on $\pa D$ using the spatial Fourier transform $\mathcal{F}$ in $\mathcal{S}'(\R^2)$, local charts $\varPhi_j\!:\R^2\rightarrow \pa D$, and the associated partition of unity $\chi_j\!:\pa D\rightarrow\R, \, j=1,...,N$, by
\ben
\|\bphi(\cdot,s)\|^2_{H^{\pm 1/2}_s(\pa D)^3}:=\sum^N_{j=1}\int_{\R^2}\left(|s|^2+|\bx|^2\right)^{\pm 1/2}
\left|\mathcal{F}[(\chi_j\bphi)\circ\varPhi_j(\bx)]\right|^2d\bx,\quad s\in\R+i\sigma,\, \sigma>0.
\enn
When equipped with these norms, the spaces $H^{\pm 1/2}_s(\pa D)^3$ are dual to each other for the duality product extending the $L^2$ inner product $\left<\xxs \text{\bf f},\text{\bf g} \xxs \right>_{\pa D}=\int_{\pa D}\exs \ov{\text{\bf g}}\nxs\cdot\nxs \text{\bf f} \,\exs dS$.

Next, it is assumed that the fracture surface $\Gamma$ may be arbitrarily extended to a piecewise smooth and simply connected surface
$\pa \text{\sf{D}}$ enclosing the bounded domain $\text{\sf{D}}$ such that the normal vector $\bn$ to the fracture surface
$\Gamma$ coincides with the unit outward normal vector to $\pa \text{\sf{D}}$. Moreover, $\Gamma$ is an open set relative to $\pa \text{\sf{D}}$ with a positive surface measure. In this setting, let us define
\ben
H^{\pm 1/2}_s(\Gamma)^3&:=&\left\{\text{\bf f}\exs\big|_\Gamma \! : \,\, \text{\bf f}\in H_s^{\pm 1/2}(\pa \text{\sf{D}})^3\right\},\\
\tilde{H}^{\pm 1/2}_s(\Gamma)^3&:=&\left\{\text{\bf f}\exs\in H_s^{\pm 1/2}(\pa \text{\sf{D}})^3 : \text{supp}(\text{\bf f}\exs\xxs)\subset \overline{\Gamma}\right\}.
\enn

Given the above, note that $H^{-1/2}_s(\Gamma)^3$ and $\tilde{H}_s^{-1/2}(\Gamma)^3$ are respectively the dual spaces of $\tilde{H}_s^{1/2}(\Gamma)^3$ and $H^{1/2}_s(\Gamma)^3$ so that the following embeddings hold
\ben
\tilde{H}_s^{1/2}(\Gamma)^3\subset {H}_s^{1/2}(\Gamma)^3 \subset L^2_s(\Gamma)^3
\subset \tilde{H}_s^{-1/2}(\Gamma)^3 \subset H^{-1/2}_s(\Gamma)^3.
\enn
\begin{remark}
For brevity, a short-hand notation is used in what follows for the vector norms such that e.g.,~$\|\cdot\|_{H_s^{1/2}(\Gamma)^3}$ is implied by $\|\cdot\|_{H_s^{1/2}(\Gamma)}$.
\end{remark}

In the context of~\eqref{eq1}, given $\widehat{\bv}\in H^1_s(\R^3\backslash\Gamma)^3$ then, by the trace theorem, {$[\![\widehat{\bv}]\!]\in \tilde{H}_s^{1/2}(\Gamma)^3$}~\cite{McLean2000}. Let us define by ${\rm tr}_\Gamma\! : \widehat{\bv}\rightarrow [\![\widehat{\bv}]\!]$ the trace operator from $H^1(\R^3\backslash\Gamma)^3$ into $\tilde{H}^{1/2}(\Gamma)^3$.

\begin{lemma}\label{lemma1}
	Let $\sigma_0>0$, there exists a constant $C>0$ depending only on $\Gamma$ and $\sigma_0$ such that
	\ben
	\|{\rm tr}_\Gamma \widehat{\bv}\|_{\tilde{H}_s^{1/2}(\Gamma)}\leq C\xxs\|\widehat{\bv}\|_{H_s^1(\R^3\backslash\Gamma)}\quad \forall\exs\widehat{\bv}\in H^1(\R^3\backslash\Gamma)^3, \,\,~s\in\mathbb{C}_{\sigma_0}.
	\enn
\end{lemma}	

Now, we are in position to investigate the well-posedness of the direct scattering problem \eqref{eq2}. This problem can be written variationally in terms of $\widehat\bv\in H^1(\R^3\backslash\Gamma)^3$ as
\be\label{eq3}
A(\widehat\bv,\bw)~=~g(\bw),\quad \forall \bw\in H^1(\R^3\backslash\Gamma)^3,
\en
with
\be\label{eq4}
A(\widehat\bv,\bw)~:=\,- s^2\int_{\R^3\backslash\Gamma}\ov\bw\cdot\widehat\bv \exs d\bx
\,+\int_{\R^3\backslash\Gamma}\nabla\ov\bw :\bC :\nabla\widehat\bv \exs d\bx
\xxs\,+\xxs\left<\xxs{\bf K}\xxs [\![\widehat\bv]\!],[\![\bw]\!]\xxs\right>_{\Gamma},
\en
where $\langle\exs\cdot\exs,\exs\cdot\exs\rangle_\Gamma$ denotes the duality product $\langle H^{-1/2}(\Gamma),\tilde{H}^{1/2}(\Gamma)\rangle$, and
\ben
g(\bw)~:=~\int_\Gamma \exs [\![\ov\bw]\!] \cdot\xxs \wh\bt^i \exs dS.
\enn
Then we have the following result.
\begin{theorem}\label{thm2}
	Given the symmetric, real-valued, and positive semi-definite stiffness matrix ${\bf K}\in L^\infty(\Gamma)^{3\times 3}$, let $s\in\mathbb{C}_{\sigma_0}$ for $\sigma_0>0$ and assume that $\wh\bt^i\in H^{-1/2}(\Gamma)^3$. Then,~\eqref{eq2} has a unique solution $\widehat\bv(\cdot,s)\in H^1(\R^3\backslash\Gamma)^3$. Moreover, there exists a constant $C$ depending only on $\sigma_0$ and $\Gamma$ such that
	\be\label{eq5}
	\|\widehat\bv(\cdot,s)\|_{H^1_s(\R^3\backslash\Gamma)}\leq C(\sigma_0,\Gamma)|s| \|\wh\bt^i(\cdot,s)\|_{H^{-1/2}_s(\Gamma)}.
	\en
\end{theorem}
\begin{proof}
	As mentioned earlier, $\widehat\bv$ solves \eqref{eq2} if and only if \eqref{eq3} is satisfied. Multiplying $A(\widehat\bv,\bw)$ defined in~\eqref{eq4} by $i\bar{s}:=i\eta+\sigma$, taking the real part, and setting $\bw=\widehat\bv$, one obtains
	\ben
	\Re\left(i\bar{s}A(\widehat\bv,\widehat\bv)\right) &=& \sigma\big(\int_{\R^3\backslash\Gamma}|s\widehat\bv|^2\exs d\bx \,+\, \int_{\R^3\backslash\Gamma} \nabla \ov{\widehat\bv}:\bC:\nabla \widehat\bv \exs d\bx \,+\, \int_\Gamma [\![\ov{\widehat\bv}]\!]\cdot{\bf K}\xxs [\![\widehat\bv]\!] \exs dS \big) \\*[0.5mm]
	&\geq& C {\sigma_0} \|\widehat\bv\|^2_{H^1_s(\R^3\backslash\Gamma)}.
	\enn
	{This shows that \eqref{eq2} admits a unique solution.} Further, since $A(\widehat\bv,\widehat\bv) = g(\widehat\bv)$ we have
	\ben
\sigma_0 \|\widehat\bv\|^2_{H^1_s(\R^3\backslash\Gamma)}
	&\leq&C \Re\left(i\bar{s}\int_\Gamma \wh\bt^i\cdot[\![\ov{\widehat\bv}]\!]\exs dS\right)\\*[0.5mm]
	&\leq&C \sqrt{|s|^2\left|\int_\Gamma \wh\bt^i\cdot[\![\ov{\widehat\bv}]\!]\exs dS\right|^2}\\*[1mm]
	&\leq& C {|s|\|\wh\bt^i\|_{H^{-1/2}_s(\Gamma)}\|[\![\widehat\bv]\!]\|_{\tilde{H}^{1/2}_s(\Gamma)}}\\*[1mm]
	&\leq& C(\sigma_0,\Gamma)|s| \|\wh\bt^i\|_{H^{-1/2}_s(\Gamma)} \|\widehat\bv\|_{H^1_s(\R^3\backslash\Gamma)}.
	\enn
The last inequality comes from Lemma \ref{lemma1}. Thus, the announced estimate \eqref{eq5} is proved.
\end{proof}

For $m\in\R$ and $\sigma\in\R$, we introduce the Hilbert space
\ben
H^m_{\sigma}(\mathbb{R};X)=\left\{f\in\mathcal{L}'_{\sigma}(\mathbb{R};X), \int_{-\infty+i\sigma}^{\infty+i\sigma}|s|^{2m}\|\widehat f(s)\|^2_X \exs ds<\infty\right\},
\enn
endowed with the norm
\be\label{eq14}
\|f\|_{H^m_{\sigma}(\mathbb{R};X)}=\left(\int_{-\infty+i\sigma}^{\infty+i\sigma}|s|^{2m}\|\widehat f(s)\|^2_X \exs ds \right)^{1/2},
\en
see, e.g., \cite{BD1986,sayas2016}. 
\begin{remark}
	For simplicity, we denote $H^{m,1}_{\sigma,\Om}$, $H^{m,\pm 1/2}_{\sigma,\Gamma}$ and $\tilde{H}^{m,\pm 1/2}_{\sigma,\Gamma}$ by $H^m_{\sigma}(\mathbb{R};X)$ with $X={H^1_s(\R^3\backslash\Gamma)^3}$, $X=H^{\pm 1/2}_s(\Gamma)^3$ and $X=\tilde{H}^{\pm 1/2}_s(\Gamma)^3$ in the rest of this paper, respectively. 
\end{remark}

As a consequence of Theorem \ref{thm2} and the use of Laplace transform, one gets
the following result.
\begin{proposition}\label{prop1}
	Let $\sigma_0>0$ and assume that $\bt^i\in H^{m+1,-1/2}_{\sigma,\Gamma}$ for some $m\in\R$. Then problem \eqref{eq1} has a unique solution $\bv\in H^{m,1}_{\sigma,\Om}$ with $\sigma\geq\sigma_0$. Moreover, there exists a constant $C$ depending only on $\sigma_0$ and $\Gamma$ such that 
	\ben
		\|\bv\|_{H^{m,1}_{\sigma,\Om}}\leq C(\sigma_0,\Gamma)\|\bt^i\|_{H^{m+1,-1/2}_{\sigma,\Gamma}}.
	\enn
	for all $\sigma\geq\sigma_0$.
\end{proposition}

Now we can define the solution operator $G$ to problem \eqref{eq1} as
\be\label{eqG}
G: H^{m+1,-1/2}_{\sigma,\Gamma} \rightarrow H^{m,1}_{\sigma,\Om}\quad\text{defined by}\quad G(\bt^i)=\bv,
\en
where $\bv\in H^{m,1}_{\sigma,\Om}$ is the unique solution of \eqref{eq1} for $\sigma>0$ and $m\in\R$. Proposition \ref{prop1} ensures that this operator is well defined and bounded.

\begin{remark}\label{PW} 
With reference to the Paley-€"Wiener theorem~\cite[Theorem 1]{H03}, the uniform bound in Theorem~\ref{thm2} with respect to $s\in\mathbb{C}_{\sigma_0}$ and the fact that if $s\mapsto\widehat{\bt}^i(\cdot,s)$  is holomorphic in $\mathbb{C}_{\sigma_0}$ with values in $H^{-1/2}(\Gamma)$ then $s\mapsto\widehat{\bv}(\cdot,s)$ is holomorphic in  $\mathbb{C}_{\sigma_0}$ with values in $L^2(\R^3\backslash\Gamma)$ implies that if $\bt^i$ is causal then the unique solution in Proposition \ref{prop1} is also causal.
\end{remark}

As a consequence of Proposition~\ref{prop1} and Remark~\ref{PW}, one observes
in particular that {if $\chi$ is a $C^{m+2}$-function with compact support then the scattering problem \eqref{eq1} has a unique solution in $H^{m,1}_{\sigma,\Omega}$ with $\sigma>0$}~since ${\bf t}^i = {\bf n}\cdot\left(\bC:\nabla {\bf u}_\chi^i \right)\in H^{m+1,-1/2}_{\sigma,\Gamma}$ for $\chi\in C^{m+2}$. Moreover, if $\chi$ vanishes for $t \leq T$ then the solution
also vanishes for $t \leq T$.  

\section{Factorization of the near-field operator}\lb{NfO}

Let $\chi\in C^{m+2}$ with $m\in\mathbb{N}$ be a smooth excitation function with compact support in time, and define $\bV_\chi\colon \!\!\!=[\bv_{\chi,1}~ \bv_{\chi,2}~ \bv_{\chi,3}]$ wherein $(\bv_{\chi,k})_{k=1,2,3}$ denotes the scattered field solving~\eqref{eq1} for the incident field $\bu_\chi^i(\bx,t;\by,{\bf{e}}_k)_{k=1,2,3}$ in \eqref{eqinc}. In this setting, given the density distribution $\bg\in H^{m,0}_{\sigma,\Gamma_i}$, the near-field operator $N^\chi$ is defined by
\be\label{eqN}
(N^\chi\bg)(\bx,t):=\int_\R\int_{\Gamma_i}\bV^\chi(\bx,t-\tau;\by)\cdot\bg(\by,\tau)d\by \xxs d\tau,\quad (\bx,t)\in\Gamma_m\times\R.
\en

From the linearity of the scattering problem with respect to the incident field, observe that $N^\chi\bg$ in \eqref{eqN} is the trace (on $\Gamma_m$) of the solution to~\eqref{eq1} with the incident $\bu_\chi^i$ replaced by the (regularized) retarded potential $L^\chi_{\Gamma_i}\bg$ as 
\be
\nonumber (L^\chi_{\Gamma_i}\bg)(\bx,t)
&=&\int_\R\int_{\Gamma_i}{\bf U}^i_\chi(\bx,t-\tau;\by)\cdot\bg(\by,\tau)d\by \xxs d\tau\\*[0.5mm]
\label{eq17}&:=&\left[\chi(\cdot)*(L_{\Gamma_i}\bg)(\bx,\cdot)\right](t),\quad (\bx,t)\in\R^3\backslash\Gamma_i\times\R,
\en
wherein ${\bf U}^i_\chi:=[\bu_{\chi,1}~ \bu_{\chi,2}~ \bu_{\chi,3}]$ which may be recast as
\ben
\begin{aligned}
{\bf U}^i_\chi(\bx,t;\by)~=~ & \frac{1}{\mu}\left(\bI_2-\nabla_{\bx}\otimes\nabla_{\bx}\right)
\frac{\chi\left(t-|\bx-\by|/\sqrt{\mu/\rho}\right)}{|\bx-\by|}\\*[0.5mm]
+~ & \frac{1}{\lambda+2\mu}\nabla_{\bx}\otimes\nabla_{\bx}\frac{\chi\left(t-|\bx-\by|/\sqrt{(\lambda+2\mu)/\rho}\right)}{|\bx-\by|},
\end{aligned}
\enn
and
\ben
(L_{\Gamma_i}\bg)(\bx,t)
&=& \int_\R\int_{\Gamma_i}\boldsymbol{\Pi}(\bx,t-\tau;\by)\cdot\bg(\by,\tau)d\by \xxs d\tau\\
&:=& \int_{\Gamma_i}\left[\boldsymbol{\Pi}(\bx,\cdot\,;\by)*\bg(\by,\cdot)\right]\nxs (t) \xxs d\by,\quad (\bx,t)\in(\R^3\backslash\Gamma_i)\times\R.
\enn
From \eqref{eq17}, since $L^\chi_{\Gamma_i}$ is a time convolution operator for the regular density $\chi$ with compact support, we have
\be\label{eq18}
\widehat{L^\chi_{\Gamma_i}\bg}(\bx,s)=\left(\widehat{L^\chi_{\Gamma_i}}(s)\widehat{\bg}(\cdot,s)\right)(\bx)
\en
for $\bx\in\R^3\backslash\Gamma_i$ and $s\in\mathbb{C}$, where
\ben
\widehat{L^\chi_{\Gamma_i}}(s)=\widehat{\chi}(s)\widehat{L_{\Gamma_i}}(s),
\enn
with $\widehat{L_{\Gamma_i}}(s)$ similar to the single layer potential in the frequency domain 
\ben
\left(\widehat{L_{\Gamma_i}}(s)\widehat{\bg}(\cdot,s)\right)\nxs(\bx)
\,=\int_{\Gamma_i}\widehat{\boldsymbol{\Pi}}(\bx,s;\by)\cdot\widehat\bg(\by,s) \xxs d\by,\quad \bx\in\R^3\backslash\Gamma_i.
\enn

\begin{assum}\label{DR}
	The pulse function $\chi:\R\rightarrow\R$ is a non-trivial and causal $C^3$-function such that its Laplace transform is holomorphic in $\mathbb{C}_0$ and has a cubic decay rate,
	\be\label{chi}
	|\widehat{\chi}(s)|\leq\frac{C}{|s|^3},\quad~s\in\mathbb{C}_0.
	\en
\end{assum}

This assumption is not strictly necessary but allows us to use relatively simple function spaces in the main result of this paper. Slower decay rates would essentially change the time regularity of all later results. It should be noted that Assumption~\ref{DR} is satisfied by causal ${C}^3$-functions with compact support.

We also need the following assumption in order to apply some unique continuation argument.

\begin{assum}\label{DRbis}
We assume that $\partial D$ that contains $Gamma$ is an analytic surface. We also assume that $\Gamma_i$ and $\Gamma_m$ are respectively parts of some analytic  boundaries of  simply connected domains $B_i$ and $B_m$ enclosing $D$. In the case $\Gamma_i=\partial B_i$ ($\Gamma_m=\partial B_m$) the boundary $\partial B_i$ ($\partial B_m$)  can be assumed to be only Lipschitz continuous.
\end{assum}
 
From now on in this section and in Section 5, Assumptions~\ref{DR} and~\ref{DRbis} are assumed to hold.

\begin{lemma}\label{slp}
	For $m\in\R$, $\sigma>0$ and $\Om_i:=\R^3\backslash\Gamma_i$, the operator 
	$
		L^\chi_{\Gamma_i} : H^{m,-1/2}_{\sigma,\Gamma_i}\rightarrow H^{m+2,1}_{\sigma,{\Om_i}} 
 	$
	is bounded and injective. Moreover, the operator
	$
		{\rm {tr}}_\Gamma L^\chi_{\Gamma_i} : H^{m,-1/2}_{\sigma,\Gamma_i}\rightarrow \tilde{H}^{m+2,1/2}_{\sigma,\Gamma}
	$
	is bounded, injective with dense range.
\end{lemma}
\begin{proof}
	For any $\bg\in H^{m,-1/2}_{\sigma,\Gamma_i}$, using the norm definition given in \eqref{eq14}, we have
	\ben
		\|L^\chi_{\Gamma_i}\bg\|^2_{H^{m+2,1}_{\sigma,\Om_i}}
		&=&\int_{-\infty+i\sigma}^{\infty+i\sigma}|s|^{2(m+2)}\|\widehat{L^\chi_{\Gamma_i}\bg}(\cdot,s)\|^2_{ H^{1}_{s}(\R^3\backslash\Gamma_i)}ds \\	
	&\leq&\int_{-\infty+i\sigma}^{\infty+i\sigma}|s|^{2m+4}{|\widehat{\chi}(s)|^2}\|\widehat{L_{\Gamma_i}}(s)		
		\widehat{\bg}(\cdot,s)\|^2_{ H^{1}_{s}(\R^3\backslash\Gamma_i)}ds.
	\enn
	In addition, given the following (see e.g.,~\cite{Kuprad1979,hsia2020})
		\ben
		\|\widehat{L_{\Gamma_i}}(s)\widehat{\bg}(\cdot,s)\|_{ H^{1}_{s}(\R^3\backslash\Gamma_i)}
	\leq C(\sigma,\Gamma)|s|\|\widehat{\bg}(\cdot,s)\|_{ H^{-1/2}_{s}(\Gamma_i)},
		\enn
		together with assumption~\eqref{chi}, one may deduce that there exists a constant $C$ such that
	\ben
		\|L^\chi_{\Gamma_i}\bg\|^2_{H^{m+2,1}_{\sigma,\Om_i}}\leq C \int_{-\infty+i\sigma}^{\infty+i\sigma}|s|^{2m-2}|s|^2\|\widehat{\bg}(\cdot,s)\|^2_{ H^{-1/2}_{s}(\Gamma_i)}ds 
		= C \|\bg\|^2_{ H^{m,-1/2}_{\sigma,\Gamma_i}}.
	\enn
	The boundedness of the operator ${\rm {tr}}_\Gamma L^\chi_{\Gamma_i}: H^{m,-1/2}_{\sigma,\Gamma_i}\rightarrow \tilde{H}^{m+2,1/2}_{\sigma,\Gamma}$ follows Lemma \ref{lemma1}.
	
	To prove the injectivity of $L^\chi_{\Gamma_i}$, suppose that $L^\chi_{\Gamma_i}\bg=\bzero$ for $\bg\in H^{m,-1/2}_{\sigma,\Gamma_i}$, then
	\ben
		\widehat{\chi}(s)\widehat{L_{\Gamma_i}}(s)\widehat{\bg}(s,\cdot)=\bzero\quad\text{in}~\R^3\backslash\Gamma_i~\text{for a.e.}~s\in\R+i\sigma.
	\enn
	Our assumptions imply in particular that the zeros of $\widehat{\chi}(s),s\in\R+i\sigma$ form an at most
	countable discrete set without finite accumulation point. Hence, $\widehat{L_{\Gamma_i}}(s)\widehat{\bg}(s,\cdot)=\bzero$ in $\R^3\backslash\Gamma_i$
	for a.e. $s\in\R+i\sigma$.	
	Since {the operator $\widehat{L_{\Gamma_i}}(s):H^{-1/2}(\Gamma_i)\rightarrow H^1(\R^3\backslash\Gamma_i)$ is injective for all $s\in\R+i\sigma$ (see~\cite{Kuprad1979})}, we obtain $\widehat{\bg}(s,\cdot)=\bzero$ on $\Gamma_i$ for a.e. $s\in\R+i\sigma$ which implies that $\bg=\bzero$. The injectivity of ${\rm {tr}}_\Gamma L^\chi_{\Gamma_i}$ can be proved in a similar way using the injectivity of the operator ${\rm {tr}}_\Gamma \widehat{L^\chi_{\Gamma_i}}(s)$ for a.e.~$s\in\R+i\sigma$ {(see~\cite{HLM14})}.
	
	Finally, the denseness of the range of ${\rm {tr}}_\Gamma L^\chi_{\Gamma_i}$ can be seen by showing that the $L^2$-adjoint of ${\rm {tr}}_\Gamma L^\chi_{\Gamma_i}$ is injective. Let $A^*$ denote the adjoint of an operator $A$, then
	\ben
	\left({\rm {tr}}_\Gamma L^\chi_{\Gamma_i}\right)^*
	=\,\exs\chi(-\xxs\cdot\xxs)*\left({\rm {tr}}_\Gamma L_{\Gamma_i}\right)^*\,:=\,B,
	\enn
	where
	\ben
	\left[\left({\rm {tr}}_\Gamma L_{\Gamma_i}\right)^*{\bm h}\right](\bx,t)
	\,:={\int_{\Gamma}{\left[\boldsymbol{\Pi}(\bx,-\xxs\cdot\xxs;\by)*{\bm h}(\by,\cdot)\right](t)}d\by},\quad(\bx,t)\in\Gamma_i\times\R.
	\enn
	The injectivity of $\left({\rm {tr}}_\Gamma L^\chi_{\Gamma_i}\right)^*$ on $\left(\tilde{H}^{m+2,1/2}_{\sigma,\Gamma}\right)'\!=H^{-m-2,-1/2}_{-\sigma,\Gamma}$ can be also checked by analyzing the injectivity of the Laplace transform of $\left({\rm {tr}}_\Gamma L^\chi_{\Gamma_i}\right)^*$ on the $\R-i\sigma$ line. 	
	Now, in light of~\eqref{eq18}, define
	\ben
		\widehat{B}(s)=\widehat{\chi}(-s) \xxs {\rm tr}_{\Gamma_i}\widehat{L_\Gamma}(-s),
	\enn
	so that following a similar argument used for the injectivity of ${\rm tr}_\Gamma L^\chi_{\Gamma_i}$, one obtains the desired result as a consequence of the injectivity of ${\rm tr}_{\Gamma_i}\widehat{L_\Gamma}(-s):H^{-1/2}(\Gamma_i)\rightarrow {\tilde H}^{1/2}(\Gamma)$
	for a.e. $s\in\R-i\sigma$ {(see~\cite{HLM14})} which completes the proof.
\end{proof}

\begin{remark}
	Note that in Lemma~\ref{slp}, the space $H^{m,-1/2}_{\sigma,\Gamma_i}$ can be replaced by $H^{m,0}_{\sigma,\Gamma_i}$ since the latter space is continuously embedded in the former one. 
\end{remark}
For $m\in\R$ and $\sigma>0$, denote
\ben
X^m_{\sigma,\Omega}:=\left\{\bu\in H^{m,1}_{\sigma,\Om} : \nabla\cdot\left(\bC:\nabla \bu\right)-\ddot{\bu} = \bzero \quad\text{in}\quad \R^3{\backslash\Gamma}\times\R\right\}, 
\enn
where the differential equation holds in the distributional sense. Then we introduce the free-field traction operator $T_\Gamma:X^m_{\sigma,\Omega}\rightarrow H^{m,-1/2}_{\sigma,\Gamma}$
\ben
T_\Gamma\bu = {\bf n}\cdot\left(\bC:\nabla {\bf u} \right)\big|_\Gamma.
\enn

\begin{lemma} \label{T}
	For $m\in\R$ and $\sigma>0$, the operator 
	$T_\Gamma : X^m_{\sigma,\Omega}\rightarrow H^{{m},-1/2}_{\sigma,\Gamma}$
	is bounded when $X^m_{\sigma,\Omega}$ is equipped with the norm on $H^{m,1}_{\sigma,\Om}$. 
\end{lemma}
\begin{proof}
	For any ${\bf u}\in X^m_{\sigma,\Omega}$, we have
	\ben
		\|T_\Gamma\bu\|^2_{H^{{m},-1/2}_{\sigma,\Gamma}}
	&=&\int_{-\infty+i\sigma}^{\infty+i\sigma}|s|^{2{m}}\|\widehat{T_\Gamma\bu}(\cdot,s)\|^2_{H_s^{-1/2}(\Gamma)}ds\\
	&=&\int_{-\infty+i\sigma}^{\infty+i\sigma}|s|^{2{m}}\|{\bf n}\cdot\left(\bC:\nabla \widehat{\bf u}(\cdot,s) \right)\|^2_{H_s^{-1/2}(\Gamma)}ds.
	\enn
	For $s\in \R+i\sigma$, ${\bf p}\in L^2(\Om)^3$ such that $\rm{div}~ {\bf p}\in L^2(\Om)$ there exists a constant $C$ independent from $s$ such that (see \cite[Proposition 9]{HLM14})
	\ben
	\|{\bf n}\cdot {\bf p}\|_{H^{-1/2}_s(\Gamma)}\leq C\left(\|{\bf p}\|_{L^2(\Om)^3}+\|\rm{div}~ {\bf p}/|s|\|_{L^2(\Om)}\right).
	\enn
	Therefore, for ${\bf u}\in X^m_{\sigma,\Omega}$, using $\nabla\cdot\left(\bC:\nabla \widehat{\bf u}(\cdot,s)\right)=- s^2 \widehat{\bf u} (\cdot,s)$ in $\Om$, we deduce
	\ben
	\|{\bf n}\cdot\left(\bC:\nabla \widehat{\bf u}(\cdot,s) \right)\|^2_{H_s^{-1/2}(\Gamma)}
	&\leq& C\left(\|\bC:\nabla \widehat{\bf u}(\cdot,s)\|_{L^2(\R^3\backslash\Gamma)^3}+\|\nabla\cdot\left(\bC:\nabla \widehat{\bf u}(\cdot,s)\right)/|s|\|_{L^2(\R^3\backslash\Gamma)} \right)\\
	&\leq& C{\|\widehat{\bf u}(\cdot,s)\|^2_{H_s^{1}(\R^3\backslash\Gamma)}}.	
	\enn
	Thus, we have 
	\ben
		\|T_\Gamma\bu\|^2_{H^{{m},-1/2}_{\sigma,\Gamma}}
	\leq C \int_{-\infty+i\sigma}^{\infty+i\sigma}|s|^{2{m}}{
		\|\widehat{\bf u}(\cdot,s)\|^2_{H_s^{1}(\R^3\backslash\Gamma)}}ds
	= C \|\bu\|^2_{H^{{m},1}_{\sigma,\Om}}.
	\enn
	This completes the proof.
\end{proof}

\begin{lemma}\label{lemma10}
	Let $m\in\R$ and $\sigma>0$. If ${\bf w} \in X^m_{\sigma,\Om}$, then there exists a sequence $(\bpsi_n)_{n\in\mathbb{N}}$ in $H^{m-2,0}_{\sigma,\Gamma_i}$ such that
		\ben
			L^\chi_{\Gamma_i}\bpsi_n\rightarrow {\bf w}~\text{in}~H^{m-1,1}_{\sigma,\Om}~\text{as}~n\rightarrow \infty.
		\enn
\end{lemma}
\begin{proof}
	Let ${\bf w}\in X^m_{\sigma,\Om}$, then from Lemma \ref{lemma1} we have ${\rm tr}_\Gamma{\bf w}\in \tilde{H}^{m,1/2}_{\sigma,\Gamma}$. Lemma \ref{slp} implies that there exists a sequence $(\bpsi_n)_{n\in\mathbb{N}}$ in $H^{m-2,0}_{\sigma,\Gamma_i}$ such that
	\ben
		{\rm tr}_\Gamma L^\chi_{\Gamma_i}(\bpsi_n)\rightarrow {\rm tr}_\Gamma{\bf w}~\text{in}~ \tilde{H}^{m,1/2}_{\sigma,\Gamma}~\text{as}~n\rightarrow\infty.
	\enn
	Note that both $L^\chi_{\Gamma_i}\bpsi_n$ and ${\bf w}$ solve the homogeneous Navier equation in $\Om$. The bounds of Theorem \ref{thm2} combined with Plancherel's identity imply that 
	\ben
		\|L^\chi_{\Gamma_i}\bpsi_n-{\bf w}\|_{H^{m-1,1}_{\sigma,\Om}}
		\leq C \|{\rm tr}_\Gamma L^\chi_{\Gamma_i}(\bpsi_n)-{\rm tr}_\Gamma{\bf w}\|_{\tilde{H}^{m,1/2}_{\sigma,\Gamma}}\rightarrow 0~\text{\,as\,}~n\rightarrow 0.
	\enn
\end{proof}

\begin{proposition}\label{prop3}
	Let $m\in\R$ and $\sigma>0$, then the operator $T_\Gamma L^\chi_{\Gamma_i}: H^{m,-1/2}_{\sigma,\Gamma_i}\rightarrow H^{m+1,-1/2}_{\sigma,\Gamma}$ is bounded, injective and has dense range.
\end{proposition}
\begin{proof}
	Boundedness follow from the boundedness of $L^\chi_{\Gamma_i}:H^{m,-1/2}_{\sigma,\Gamma_i}\rightarrow H^{m+2,1}_{\sigma,\Om_i}$ and the boundedness of $T_\Gamma: X^{m+2}_{\sigma,\Om_i}\rightarrow H^{m+1,-1/2}_{\sigma,\Gamma}$.
	
	For the injectivity, we assume that $T_\Gamma L^\chi_{\Gamma_i}\bpsi=\bzero$ for some $\bpsi\in H^{m,-1/2}_{\sigma,\Gamma_i}$. Then \cite[Lemma 5.3]{Fatemeh2017} implies 
	$\widehat{L^\chi_{\Gamma_i}}(s)\widehat\bpsi(\xxs\cdot\xxs,s)=\bzero$ in $\R^3\backslash\Gamma$ for a.e.~$s\in\R+i\sigma$. We conclude, as in Lemma \ref{slp}, that $\widehat{L_{\Gamma_i}}(s)\widehat\bpsi(\xxs\cdot\xxs,s)=\bzero$ in $\R^3\backslash\Gamma$ for a.e. $s\in\R+i\sigma$. The jump relation for $\widehat{L_{\Gamma_i}}(s)$ and the injectivity of the single-layer operator (see \cite{Kuprad1979}) shows that $\bpsi=\bzero$.
	
	To prove the denseness of the range of $T_\Gamma L^\chi_{\Gamma_i}$, consider $\bts\in H^{m+1,-1/2}_{\sigma,\Gamma}$. Since the embedding $H^{m+3,-1/2}_{\sigma,\Gamma}$ into $H^{m+1,-1/2}_{\sigma,\Gamma}$ is dense, there exists a sequence $(\bts_n)_{n\in\mathbb{N}}\subset H^{m+3,-1/2}_{\sigma,\Gamma}$ such that $\bts_n\rightarrow\bts$ in $H^{m+1,-1/2}_{\sigma,\Gamma}$. Due to Proposition \ref{prop1}, there exists $\bv_n\rightarrow H^{m+2,1}_{\sigma,\Om}$ such that $\bv_n$ satisfied \eqref{eq1} with $\bt^i=\bts_n$. Lemma \ref{lemma10} states that we can approximate $\bv_n$ by potentials $X^{m+2}_{\sigma,\Om}\ni L^\chi_{\Gamma_i}\bpsi_{n,l}\rightarrow \bv_n$ as $l\rightarrow\infty$ in $H^{m+2,1}_{\sigma,\Om}$ with $(\bpsi_{n,l})_{l\in\mathbb{N}}\subset H^{m,-1/2}_{\sigma,\Gamma_i}$. Finally, the continuity of $T_\Gamma$ from $X^{m+2}_{\sigma,\Om}$ into $H^{m+1,-1/2}_{\sigma,\Gamma}$ shows that 
	\ben
		T_\Gamma L^\chi_{\Gamma_i}\bpsi_{n,l}\rightarrow T_\Gamma \bv_n=\bts_n\quad\text{as}~l\rightarrow \infty~\text{in}~H^{m+1,-1/2}_{\sigma,\Gamma}.
	\enn
	Since, by construction, $\bts_n\rightarrow\bts$ as $n\rightarrow\infty$ in $H^{m+1,-1/2}_{\sigma,\Gamma}$ the proof is complete.
\end{proof}

For $m\in\R$ and $\sigma>0$, we introduce a restriction of this operator to $\Gamma_m$,
\ben
G_{\Gamma_m}:H^{m+1,-1/2}_{\sigma,\Gamma}\rightarrow \tilde{H}^{m,1/2}_{\sigma,\Gamma_m}
\quad\text{defined by}\quad G_{\Gamma_m}(\bt^i)={\rm tr}_{\Gamma_m}G(\bt^i)\big.
\enn
where the solution operator $G$ is defined in~\eqref{eqG}. The well posedness for the forward problem in Proposition \ref{prop1} and the trace theorem in Lemma \ref{lemma1} ensure that this operator is well defined and bounded.

\begin{lemma} \label{gm}
	Let $m\in\R$ and $\sigma>0$, the operator $G_{\Gamma_m}:H^{m+1,-1/2}_{\sigma,\Gamma}\rightarrow \tilde{H}^{m,1/2}_{\sigma,\Gamma_m}$ is injective with dense range.
\end{lemma}
\begin{proof}
	Let $\bv=G(\bt^i)$ for $\bt^i\in H^{m+1,-1/2}_{\sigma,\Gamma}$. Assume that 
	$G_{\Gamma_m}(\bt^i)=\bzero$. Then, due to our assumptions on $\Gamma_m$ that is either a closed Lipschitz surface or an analytic open surface, the unique continuation property and unique solvability of exterior scattering problems at complex frequencies in $\R + i\sigma$ imply that $\widehat{\bv}(\cdot,s) = \bzero$ in $\Omega$ for a.e. $s\in\R+i\sigma$. This implies that $\bt^i = -T_\Gamma \bv+K[\![\bv]\!]=\bzero$ by Lemmas~\ref{lemma1} and~\ref{T}.
	
	Now we prove the denseness of the range of $G_{\Gamma_m}$. We observe that the range of $G_{\Gamma_m}$ contains ${\rm tr}_{\Gamma_m}{\bf u}$ where
	\ben
		{\bf u}(\bx,t):=(L_{\Gamma}\bg)(\bx,t)=\int_\R\int_{\Gamma}\boldsymbol{\Pi}(\bx,t-\tau;\by)\cdot\bg(\by,\tau)d\by \tau,\quad(\bx,t)\in\R^3\backslash\Gamma\times\R
	\enn
	for some density $\bg\in  H^{m+1,-1/2}_{\sigma,\Gamma}$. This simply comes from the
	fact that ${\bf u}:=L_{\Gamma}\bg\in H^{m,1}_{\sigma,\Omega}$ for $\bg\in  H^{m+1,-1/2}_{\sigma,\Gamma}$ (see Lemma \ref{slp}) and it can be written as ${\bf{u}}=G(\bt^i)$ with $\bt^i:=T_\Gamma L_{\Gamma}\bg\in H^{m+1,-1/2}_{\sigma,\Gamma}$. Following the last part of the proof of Lemma \ref{slp}, one may show that 
	\ben
		{\rm tr}_{\Gamma_m} L_\Gamma:H^{m+1,-1/2}_{\sigma,\Gamma}\rightarrow \tilde{H}^{m,1/2}_{\sigma,\Gamma_m}
	\enn 
	has dense range. This concludes the proof.
\end{proof}

For $\by\in\Gamma_i$, we consider the incident field $\bu^i_\chi$ given in \eqref{eqinc} which belongs to $X^{m+1,1}_{\sigma,\Om}$. Due to Lemma~\ref{T}, we infer that
\ben
\bt^i:=T_\Gamma \bu^i_\chi\in H^{{m+1},-1/2}_{\sigma,\Gamma}.
\enn
In consequence, Proposition \ref{prop1} implies that the scattered field $\bv(\cdot,\cdot;\by)$ is well defined in $H^{m,1}_{\sigma,\Om}$ and the trace theorem
\ref{lemma1} implies that ${\rm tr}_{\Gamma_m}\bv(\cdot,\cdot;\by)$ is well defined in
$\tilde{H}^{m,1/2}_{\sigma,\Gamma}$. Since ${\rm tr}_{\Gamma_m}\bv(\cdot,\cdot;y)=G_{\Gamma_m}T_\Gamma \bu^i_\chi$, the linear combination of several incident pulses produces the corresponding linear combination of the measurements. Therefore, for regular densities $\bg$, the near-field operator $N^\chi$ simply satisfies
\ben
(N^\chi\bg)(\bx,t)
&=&\int_\R\int_{\Gamma_i}\bV^\chi(\bx,t-\tau;\by)\cdot\bg(\by,\tau)d\by d\tau\\
&=& G_{\Gamma_m}T_\Gamma\left(\int_\R\int_{\Gamma_i}{\bf U}^\chi(\cdot,\cdot-\tau;\by)\cdot\bg(\by,\tau)d\by d\tau\right)(\bx,t)\\
&=& G_{\Gamma_m}T_\Gamma L^\chi_{\Gamma_i}\bg(\bx,t),\quad (\bx,t)\in\Gamma_m\times\R.
\enn
Therefore, the factorization of $N^\chi$ can be written as
\ben
N^\chi = G_{\Gamma_m}T_\Gamma L^\chi_{\Gamma_i}.
\enn
\begin{proposition} \label{prop2}
	Let $m\in\R$ and $\sigma>0$, the operator $N^\chi$ is bounded, injective and has dense range from $H^{m,-1/2}_{\sigma,\Gamma_i}$ to $\tilde{H}^{m,1/2}_{\sigma,\Gamma_m}$. 
\end{proposition}
\begin{proof}
	This follows from the factorization $N^\chi = G_{\Gamma_m}T_\Gamma L^\chi_{\Gamma_i}$ and Lemma \ref{slp}, \ref{T} and \ref{gm}.
\end{proof}
\begin{remark}
	With the notation $L^2_{\sigma,\Gamma}=H^{0,0}_{\sigma,\Gamma}$, Proposition \ref{prop2} in particular implies that for all $\sigma>0$,
	\ben
		N^\chi: L^2_{\sigma,\Gamma_i}\rightarrow L^2_{\sigma,\Gamma_m}
	\enn 
	is bounded and injective with dense range.
\end{remark}

\section{Inverse solution}\lb{INS}
The main idea is to construct an approximate solution to the near field equation
\be\label{eq12}
(N^\chi\bg_L)(\bx,t) \,=\, {\bm\phi}^\bzeta_L(\bx,t; {\bf d},t_\circ),\quad (\bx,t)\in\Gamma_m\!\times\R,
\en
where ${\bm\phi}^\bzeta_L$ can be interpreted as a trial radiating field affiliated with the admissible density $\bzeta\in \tilde{H}^{{m},1/2}_{\sigma,L}$ specified 
over a smooth, non-intersecting trial fracture $L$ given by
\be\label{eq13}
{\bm\phi}^\bzeta_L(\bx,t; {\bf d},t_\circ)\exs:=\int_\R\int_L \boldsymbol{T}(\bx,\tau;\by,t_\circ){\bf d}\cdot\bzeta(t-\tau,\by) \exs d\by \xxs d\tau\quad (\bx,t)\in\R^3\backslash L\times\R,
\en
where ${\bf d}\in\R^3$ is the polarization direction, and $t_\circ$ represents a starting time. Moreover, the fundamental normal traction $\boldsymbol{T}$ is defined by
\ben
\boldsymbol{T}(\bx,t;\by,t_\circ):= \partial\Pi(\bx,t;\by,t_\circ)/{\partial \bm n}(\by)\quad (\by,t)\in L\times\R.
\enn

 The fundamental theorem of linear sampling indicates that the norm of $\bg_L$ is unbounded when $L\not\subset\Gamma$. Hence, one can construct an image of the hidden fracture $\Gamma$ by plotting $L\mapsto 1/\|\bg_L\|$ in the sampling region. 
\begin{theorem} 
	Let $\sigma>0$, $t_\circ\in\R$, ${\bf d}\in\R^3$, and some density $\bzeta\in \tilde{H}^{{m},1/2}_{\sigma,L}$, then
	\begin{enumerate}
		\item 
		For $L\subset\Gamma$, there exists a density vector $\bg_{L,\epsilon}\in L^2_{\sigma,\Gamma_i}$ such that 
		$\|N^\chi\bg_{L,\epsilon}-{\bm\phi}^\bzeta_L\|_{L^2_{\sigma,\Gamma_m}}\!\nxs\leq\epsilon$ and 
		$\lim_{\epsilon\rightarrow 0}\|T_\Gamma L^\chi_{\Gamma_i}\bg_{L,\epsilon}\|_{H^{1,-1/2}_{\sigma,\Gamma}}<\infty.$
		\item 
		For $L\not\subset\Gamma$, for all density vectors $\bg_{L,\epsilon}\in L^2_{\sigma,\Gamma_i}$ such that
		$\|N^\chi\bg_{L,\epsilon}-{\bm\phi}^\bzeta_L\|_{L^2_{\sigma,\Gamma_m}}\!\!\leq\epsilon$, one has 
		$\lim_{\epsilon\rightarrow 0}\|T_\Gamma L^\chi_{\Gamma_i}\bg_{L,\epsilon}\|_{H^{1,-1/2}_{\sigma,\Gamma}}=\infty.$
	\end{enumerate}
\end{theorem}
\begin{proof}
	Assume that $L\subset\Gamma$, then ${\bm\phi}^\bzeta_L\in H^{m,1}_{\sigma,\R^3\backslash L}\subset H^{m,1}_{\sigma,\Om}$ for  $\bzeta\in \tilde{H}^{{m},1/2}_{\sigma,L}$ by definition \eqref{eq13} and the property of the double-layer potential given in \cite{hsia2020,Kuprad1979}. By extending $\bzeta$ from $L$ to $\Gamma$ through zero padding, we have $\bzeta\in \tilde{H}^{{m},1/2}_{\sigma,\Gamma}$. From the well-posedness of the forward scattering problem {and the fact that $T_\Gamma(\bt^i)=[\![\bv]\!]=\bzeta$ from $H^{{m-1},-1/2}_{\sigma,\Gamma}$ to $\tilde{H}^{{m},1/2}_{\sigma,\Gamma}$ has a bounded inverse (see \cite[Lemma 5.6]{Fatemeh2017})}, we know that ${\bm\phi}^\bzeta_L$ is the unique causal solution to problem \eqref{eq1} with the boundary data $\bt^i_L= {T_\Gamma\nxs}^{-1} \bzeta\in  H^{{m-1},-1/2}_{\sigma,\Gamma}$  and
	\ben
		G_{\Gamma_m}\bt^i_L \,=\, {\bm\phi}^\bzeta_L\quad\text{on}~\Gamma_m\nxs\times\R.
	\enn
	Thus, one may approximate $\bt^i_L$, thanks to the denseness of the range of 
	 $T_\Gamma L^\chi_{\Gamma_i}$ given in Proposition~\ref{prop3}, such that for $\epsilon>0$, there exists $\bg_{L,\epsilon}\in L^2_{\sigma,\Gamma_i}$ such that
	\ben
	\| T_\Gamma {L_{\Gamma_i}^\chi}\bg_{L,\epsilon}-\bt^i_L \|_{H^{{m-1},-1/2}_{\sigma,\Gamma}}\leq\epsilon\quad\land\quad \lim_{\epsilon\rightarrow 0}\| T_\Gamma L^\chi_{\Gamma_i}\bg_{L,\epsilon}\|_{H^{m-1,-1/2}_{\sigma,\Gamma}}<\infty.
	\enn
	{The continuity of $G_{\Gamma_m}$ from $H^{m+1,-1/2}_{\sigma,\Gamma}$ into $\tilde{H}^{m,1/2}_{\sigma,\Gamma_m}$} implies that
	\ben
		\|N^\chi\bg_{L,\epsilon}-{\bm\phi}^\bzeta_L\xxs\|_{L^2_{\sigma,\Gamma_m}}
	\!\leq\, C\|N^\chi\bg_{L,\epsilon}-{\bm\phi}^\bzeta_L\xxs \|_{\tilde{H}^{0,1/2}_{\sigma,\Gamma_m}}
	\exs\leq\, C\| T_\Gamma L_{\Gamma_i}^\chi \bg_{L,\epsilon}-\bt^i_L \|_{H^{1,-1/2}_{\sigma,\Gamma}}
	\leq C \epsilon.
	\enn
	
	Now consider the case $L\not\subset\Gamma$, we argue by contradiction and assume that there is a positive sequence $(\epsilon_n)_{n\in\mathbb{N}}$ and suppose that there exists $C>0$ such that  
	\be\label{eq15}
		\|T_\Gamma L^\chi_{\Gamma_i}\bg_{L,\epsilon_n}\|_{H^{1,-1/2}_{\sigma,\Gamma}}\leq C.
	\en
	Hence, there is a weakly convergent subsequence $\bt^i_n:= T_\Gamma L^\chi_{\Gamma_i}\bg_{L,\epsilon_n}$ that weakly converges in $H^{1,-1/2}_{\sigma,\Gamma}$ to some $\bt^i\in H^{1,-1/2}_{\sigma,\Gamma}$. Now, let us set
	\ben
		\bv\,=\,G \exs \bt^i\in H^{0,1}_{\sigma,\Omega}.
	\enn
	Since $\bt^i_n\rightarrow \bt^i$ weakly in $H^{1,-1/2}_{\sigma,\Gamma}$, the factorization of $N^\chi$ implies that $N^\chi \bg_{L,\epsilon_n}\!\rightarrow\, {\rm tr}_{\Gamma_m}\bv$ in $L^2_{\sigma,\Gamma_m}$ as $n\rightarrow\infty$ due to Proposition \ref{prop2}. 
	Since $\|N^\chi\bg_{L,\epsilon}-{\bm\phi}^\bzeta_L\xxs\|_{L^2_{\sigma,\Gamma_m}}\!\leq\epsilon$, we have
	$\bv={\bm\phi}_L^\bzeta$ on $\Gamma_m\times\R$, which means that the Laplace transforms of both functions coincide:
	\ben
		\widehat{\bv}(s,\cdot)\,=\,\widehat{{\bm\phi}_L^\bzeta}(s,\cdot) \quad\text{in}~L^2(\Gamma_m)~\text{for a.e.}~s\in\R+i\sigma.
	\enn
	Both $\widehat{\bv}(s,\cdot)$ and $\widehat{{\bm\phi}_L^\bzeta}(s,\cdot)$ satisfy the Navier equation with complex frequency $s$ in $\R^3\backslash(\Gamma\xxs\cup\xxs L)$. Due to our assumptions on $\Gamma_m$ that is either a closed Lipschitz surface or an
	analytic open surface, the unique continuation property and unique solvability of exterior scattering problems at complex frequencies in $\R+i\sigma$ imply that $\widehat{\bv}(s,\cdot)=\widehat{{\bm\phi}_L^\bzeta}(s,\cdot)$ in $H^1(\R^3\backslash(\Gamma\cup L))$ for a.e. $s\in\R+i\sigma$. 
	Let $\Gamma\not\ni x^\circ\in L$ and let $B_r$ be a small ball centered at $x^\circ$ such that $B_r\cap\Gamma=\emptyset$. In this case, $\widehat{\bv}$ is analytic in $B_r$, while $\widehat{{\bm\phi}_L^\bzeta}$ has a discontinuity across $B_r\cap L$. This contradiction shows that our assumption \eqref{eq15} is wrong and concludes the proof.
\end{proof}

\section{Laboratory implementation} \label{exp_set}

\noindent This section makes use of the experimental data reported in~\cite{pour2021} and~\cite{Yue2021} to~(a)~examine the performance of TLSM for spatiotemporal tracking of evolving anomalies, and~(b)~conduct a comparative study of LSM-based reconstructions in time and frequency domains. Three distinct datasets are deployed in (a):~(i,ii)~waveforms collected at $75\%$ and $90\%$ of the maximum load in the post peak regime while fracturing the specimen in an MTS load frame according to~\cite{pour2021}, and~(iii)~data captured after the end of fracturing (at $60\%$ of the maximum load) where the specimen is dismounted from the load frame and ultrasonic experiments are performed according to~\cite{Yue2021}. From the latter, the multifrequency LSM reconstructions are also invoked for  the analysis in (b). To help better understand the data, a brief description of the experimental campaign in~\cite{pour2021} is provided in the sequel. Tests are sequentially conducted on a granite plate, with dimensions $0.96$m $\!\times\exs 0.3$m $\!\times\!$ 0.03m, mounted on a load frame to be fractured in the three-point bending configuration. Ultrasonic experiments are conducted at three stages (before bending starts and then while fracturing at $75\%$ and $90\%$ of the maximum load) such that the probing waves are interacting with an evolving scatterer. At each stage, in-plane shear waves of the form
\beq\lb{mat2}
\chi(t) =  H({\sf f}t) \, H(5\!-\!{\sf f}t) \, \sin\big(0.2 \pi {\sf f} t\big) \, \sin\big(2 \pi {\sf f} t\big),  \quad  {\sf f} = 30\text{KHz}, \quad  t \in (0 \, , T],
\eeq
are induced by an S-wave piezoelectric transducer at eight locations sampling $\Gamma_i$ on the specimen's boundary; here, $H(\cdot)$ is the Heaviside step function. The generated incident and total fields are then measured at 145 sensing points over the observation surface $\Gamma_m$ with the measurement period of ${T}=0.998 \text{ms}$ sampled at 1024 points. For image reconstruction (in time and frequency domains), the search area is a square of dimensions $29$cm $\nxs\!\times\!\nxs$ $29$cm in the middle of specimen discretized by a uniform grid of $100 \!\times\! 100$ points $\bz$, while the unit circle of trial normal direction $\textrm{\bf{n}}$ is sampled at 16 points. Therefore, the scattering footprints of 160000 trial dislocations $\mathcal{L}$ are used for the reconstruction in both time and frequency domains.

\subsection{Data Inversion}\label{DI}
The collected waveform data is processed as the following to compute the time-domain LSM (TLSM) maps. The latter involves four steps, namely:~(1)~assembling the scattered field ${\bf{v}}^\chi(\bx,t)$ over a unified grid in space-time,~(2)~constructing the composite near-field operator ${\boldsymbol{N}^\chi}(\bx,t)$ capturing convolution in time and multiplication in space over the source grid,~(3) computing the trial signature patterns ${\boldsymbol{\Phi}_{\boldsymbol{z},\textrm{\bf{n}}}^\bzeta}$ affiliated with $\mathcal{L}({\boldsymbol{z}}, \textrm{\bf{n}};\bzeta)$, and~(4) solving the discretized near-field equation through non-iterative minimization of the TLSM cost functional. 

\subsubsection{Scattered field in space-time}

The scattered field ${\bf{v}^{\chi}}(\bx,t)$ is computed by subtracting the free field from the total field measurements. The obtained signatures (in $x_1$ and $x_2$ directions) for multiple sources at $\by$ are then assembled as the following
\beq\lb{mat2}
{\bf{v}}^\chi (2(m-1)+1: 2(m-1)+2, k ; i) ~=\, 
\left[\begin{array}{c}
\!\!\text{v}_1\!\nxs   \\*[1mm] 
\!\!\text{v}_2\!\nxs  
\end{array}\!\right] \! (\bx_m,  t_k; \by_i ),
\eeq
for 
\beq\lb{ijl}
m = 1,2,\ldots N_m, \quad k = 1,2,\ldots N_t, \quad i = 1,2,\ldots N_i,
\eeq
wherein $N_m$, $N_t$, and $N_i$ indicating the number of samples on $\Gamma_m$, $t$, and $\Gamma_i$, respectively. 

\subsubsection{Composite near-field operator}
With reference to~\eqref{eqN}, the near-field scattering operator may be discretized as follows
\beq\lb{mat2}
\begin{aligned}
&[{\boldsymbol{N}^\chi } \bg_{\bz,\bn}^{\bzeta}](\ell, k) ~=\sum_{i=0}^{N_i}\sum_{j=0}^{k-1}{\text{v}}^\chi(\ell,k-j;i) g_{\bz,\bn}^{\bzeta}(i,j), \\
&\ell = 1,2,\ldots 2N_m, \quad  k = 1,2,\ldots N_t, \quad  i = 1,2,\ldots N_i.
\end{aligned}
\eeq
where the first summation indicates multiplication in space, while the second implies convolution in time.

\vspace{-1.5mm}
\subsubsection{Trial signatures in time}

On setting $\bzeta = \chi(t) \textrm{\bf{n}}$, every trial pair $(\bz,\textrm{\bf{n}})$ generates a unique scattering signature $\textrm{\bf{v}}_{{\bz},\textrm{\bf{n}}}({\bx},t)$ recorded at every time step $t_k \in (0 \, , T]$, $k = 1,2,\ldots N_t$, over the observation grid $\bx_m \in \Gamma_m$, $m = 1,2,\ldots N_m$, by solving
\beq\lb{PhiL2}
\begin{aligned}
&\nabla \nxs\cdot [\bC \exs \colon \! \nabla \textrm{\bf{v}}_{\bz,\textrm{\bf{n}}}](\bx,t) \,-\, \rho \exs \ddot{\textrm{\bf{v}}}_{\bz,\textrm{\bf{n}}}(\bx,t)~=~\bzero, \quad & \big(\bx \in {\mathcal{P}}\backslash \mathcal{L}, \, t \in (0, \,T]) \\*[0.75mm]
&\bn \nxs\cdot \bC \exs \colon \!  \nabla  \textrm{\bf{v}}_{\bz,\textrm{\bf{n}}}(\bx,t)~=~\bzero,  \quad & \big(\bx \in \partial{\mathcal{P}}_t, \, t \in (0, \,T]) \\*[0.75mm]
&\textrm{\bf{v}}_{\bz,\textrm{\bf{n}}}(\bx,t)~=~\bzero,  \quad & \big(\bx \in \partial{\mathcal{P}}_u, \, t \in (0, \,T]) \\*[0.75mm]
& \textrm{\bf{n}} \cdot \bC \exs \colon \!  \nabla  \textrm{\bf{v}}_{\bz,\textrm{\bf{n}}}(\bx,t) ~=~ |{\mathcal{ L}}|^{-1} \delta (\bx-\bz\!) \bzeta(t),
& \big(\bx \in \mathcal{L}, \, t \in (0, \,T])
\end{aligned}     
\eeq
where $\mathcal{P}$ represents the specimen; ${T}=0.998\text{ms}$ is the total measurement period, and $\partial{\mathcal{P}}_u$ signifies the support of three pins holding the sample in the loadframe. Simulations are performed in three dimensions via the computational platform reported in~\cite{Fatemeh2015} based on the boundary element formulation of~\eqref{PhiL2}. In this setting, the in-plane components of the computed scattered fields are recast in the following form
\beq\lb{Phi-inf-Dnum}
\bPhi_{\bz,\textrm{\bf{n}}}^\bzeta((m-1)\times 2+1:(m-1) \times 2+2, k) ~=\, 
\left[\begin{array}{c}
\!\!\text{v}^1_{\bz,\textrm{\bf{n}}}\!\nxs   \\*[1mm] 
\!\!\text{v}^2_{\bz,\textrm{\bf{n}}}\!\nxs  
\end{array}\!\right] \! ({\bx}_m, t_k ), 
\eeq
for $m = 1,2,\ldots N_m$, and $k = 1,2,\ldots N_t$. Here, $\bPhi^\bzeta_{\bz,\textrm{\bf{n}}}$ is a $2N_tN_m\!\times\! 1$ vector. 

\subsubsection{TLSM indicator}
To construct the TLSM maps, the discretized near-field equation  
\beq\lb{TLSM-C}
\begin{aligned}
&[{\boldsymbol{N}^\chi} \bg_{\bz,\textrm{\bf{n}}}^\bzeta](\bx_m, t_k)~=~\bPhi_{\bz,\textrm{\bf{n}}}^\bzeta(\bx_m, t_k), \quad \bx_m \in \Gamma_m, \,\, t_k \in (0, \,T], \\
& \qquad \qquad \quad m = 1,2,\ldots N_m, \quad k = 1,2,\ldots N_t,
\end{aligned}
\eeq 
is solved to obtain $\bg_{\bz,\textrm{\bf{n}}}^\bzeta$ for every trial pair $(\bz,\textrm{\bf{n}})$. Given the ill-posed nature of~\eqref{TLSM-C}, a regularized approximate solution $\tilde{\bg}_{\bz,\textrm{\bf{n}}}$ is obtained by minimizing the below Tikhonov cost function
\beq\label{lssm1}
\begin{aligned}
\tilde{\bg}_{\bz,\textrm{\bf{n}}} \,\,\colon \!\!\!= \,\, &\text{argmin}_{\bg_{\bz,\textrm{\bf{n}}}^\bzeta \exs\in\exs L^2(\Gamma_i)^3 \exs\times\exs L^2(0 \, T]} \, \Big\{  \\*[0.5mm]
& \norms{{\boldsymbol{N}}^\chi\bg_{\bz,\textrm{\bf{n}}}^\bzeta \,-\,\bPhi^\bzeta_{\bz,\textrm{\bf{n}}}}^2_{{L}^2(\Gamma_m)^3 \exs\times\exs L^2(0 \, T]} \,+\,\,\,  \eta_{{\bz},\textrm{\bf{n}}} \norms{{\bg_{\bz,\textrm{\bf{n}}}^\bzeta}}^2_{{L}^2(\Gamma_i)^3 \exs\times\exs L^2(0 \, T]}\Big\}.
\end{aligned}
\eeq
Here, the regularization parameter $\eta_{\bz,\textrm{\bf{n}}}$ is determined by the Morozov discrepancy principle~\cite{Kress1999}. 
The minimizer of~\eqref{lssm1} is then deployed to compute the TLSM indicator  
\beq\lb{TLSM-I}
\mathfrak{T}(\bz) \,\, = \,\, \frac{1}{\norms{\tilde{\bg}_{\bz}}_{{L}^2(\Gamma_i)^3 \exs\times\exs L^2(0 \, T]}}, \quad
\textcolor{black}{
\tilde{\bg}_{\bz} \,\,\colon \!\!\!= \,\, \text{argmin}_{\tilde{\bg}_{\bz,\textrm{\bf{n}}}} \norms{\tilde{\bg}_{\bz,\textrm{\bf{n}}}}_{{L}^2(\Gamma_i)^3 \exs\times\exs L^2(0 \, T]}},
\eeq
whereby one may also build the thresholded indicator  
\begin{equation}\label{TiFM}
\tilde{\mathfrak{T}}(\bz) \,\, := \,\, \mathbbm{1}_{\mathfrak{T}}({\bz \nxs}) \, \mathfrak{T}({\bz \nxs}), \quad {\mathbbm{1}}_{\mathfrak{T}}({\bz\nxs}) \,\, := \,\, 
\begin{cases}
1 \quad\quad\text{ if } \,\,\, \mathfrak{T}({\bz\nxs}) \,\,>\,\, \tau_{tol} \nxs\times\nxs \text{max}(\mathfrak{T})\\
0 \quad\quad\text{ otherwise}
\end{cases}\!\!\!\!\!\!,  \quad \tau_{tol} \exs\in\,\, ]0\,\,\, 1[.  
\end{equation}

In what follows, the frequency-domain LSM indicators ${\mathfrak{L}}$ and $\tilde{\mathfrak{L}}$ computed in~\cite{Yue2021} using post-fracturing waveforms, associated with $60\%$ of the maximum load, are invoked to be examined against their time-domain counterparts,~i.e.,~${\mathfrak{T}}$ and $\tilde{\mathfrak{T}}$.

\subsection{Results and discussion}\lb{RE}

The propagating fracture is periodically traced according to~\cite{pour2021} by spraying acetone on the back of specimen while being fractured in the load frame. The resulting images furnish the ground truths used to verify the time- and frequency- domain reconstructions in the sequel.

\subsubsection{Full aperture reconstruction}
%
The TLSM indicator $\mathfrak{T}$ (\emph{resp.}~$\tilde{\mathfrak{T}}$) in~\eqref{TLSM-I} (\emph{resp.}~\eqref{TiFM}) is calculated using the scattered displacements $\bv^\chi$ measured at $N_m = 145$ scanning points $\bx_m \in \Gamma_m$ on the specimen's boundary for $N_t = 1024$ uniformly distributed time steps $t_k \in (0 \,\,\, 0.998]\text{ms}$. At every testing stage, affiliated with $90\%$, $75\%$, and $60\%$ of the maximum load in the post peak~\cite{pour2021}, the transducer assumes $N_i = 8$ locations on $\Gamma_i$ implying that every sensing step entails eight independent ultrasonic experiments.  

The TLSM imaging functional~\eqref{TLSM-I} takes advantage of the rich sequential dataset and full-length waveforms in time to track the support of an advancing fracture $\Gamma = \Gamma(t)$ in space-time. Fig.~\ref{fracfull} shows the sequence of $\mathfrak{T}$ maps in the sampling region at the three loading stages mentioned above where the ground truths -- retrieved via acetone tracing in~\cite{pour2021} -- are used for verification.  

Fig.~\ref{EIF} provides a comparison between the time-domain reconstruction $\mathfrak{T}$ (\emph{resp}.~$\tilde{\mathfrak{T}}$) and its frequency-domain counterpart $\mathfrak{L}$ (\emph{resp}.~$\tilde{\mathfrak{L}}$) reported in~\cite{Yue2021}. It should be mentioned that the plots in Fig.~\ref{EIF} are affiliated with post-fracturing sensory measurements. The sharp localization and less artifacts featured in the TLSM maps could be attributed to the fact that $\mathfrak{T}$ makes use of the entire time history of data which entails less processing, whereas the multifrequency indicator $\mathfrak{L}$ deploys only the most pronounced spectral components of the measured waveforms whose interactions may be lost during signal processing and image construction. 

\begin{figure}[!tp]
\center\includegraphics[width=0.72\linewidth]{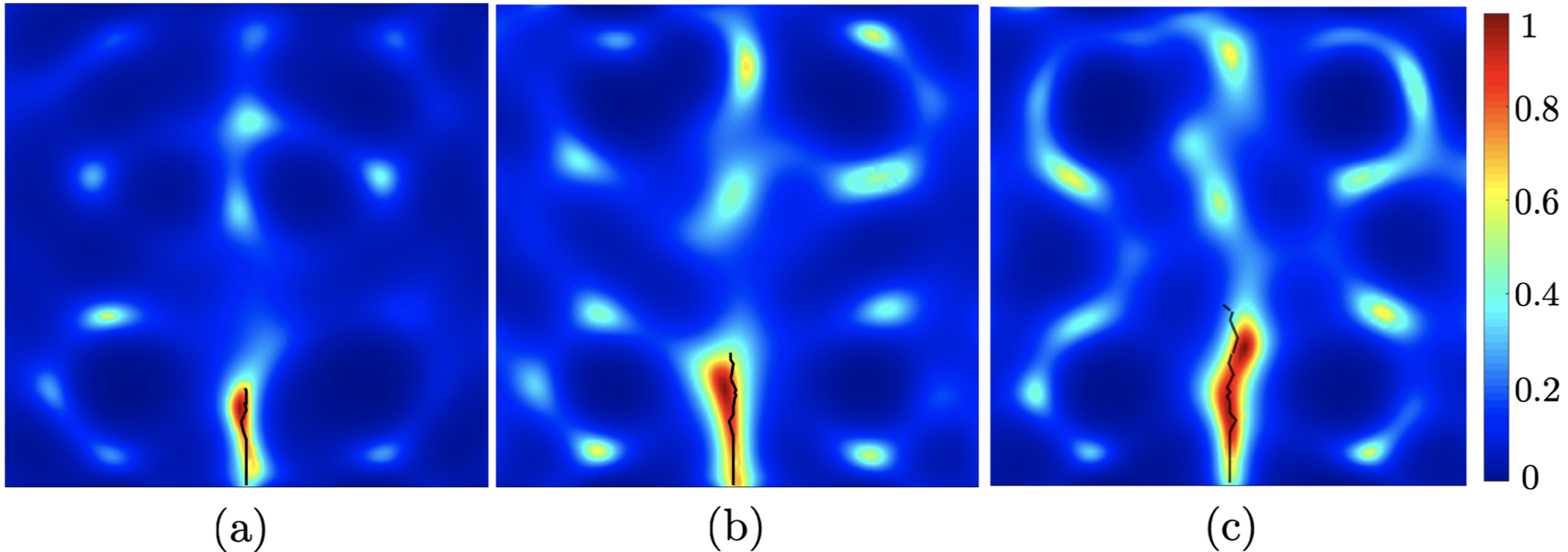} \vspace*{-2.5mm} 
\caption{\small{Time-domain reconstruction of an advancing fracture:~(a-c)~$\mathfrak{T}$ maps computed from ultrasonic waveforms measured at~$90\%$, $75\%$, and $60\%$ of the maximum load in the post-peak regime. The solid line shows the ground truth.}} \lb{fracfull}
\vspace*{-1.0mm}
\end{figure} 

\begin{figure}[!h]
\center\includegraphics[width=0.7\linewidth]{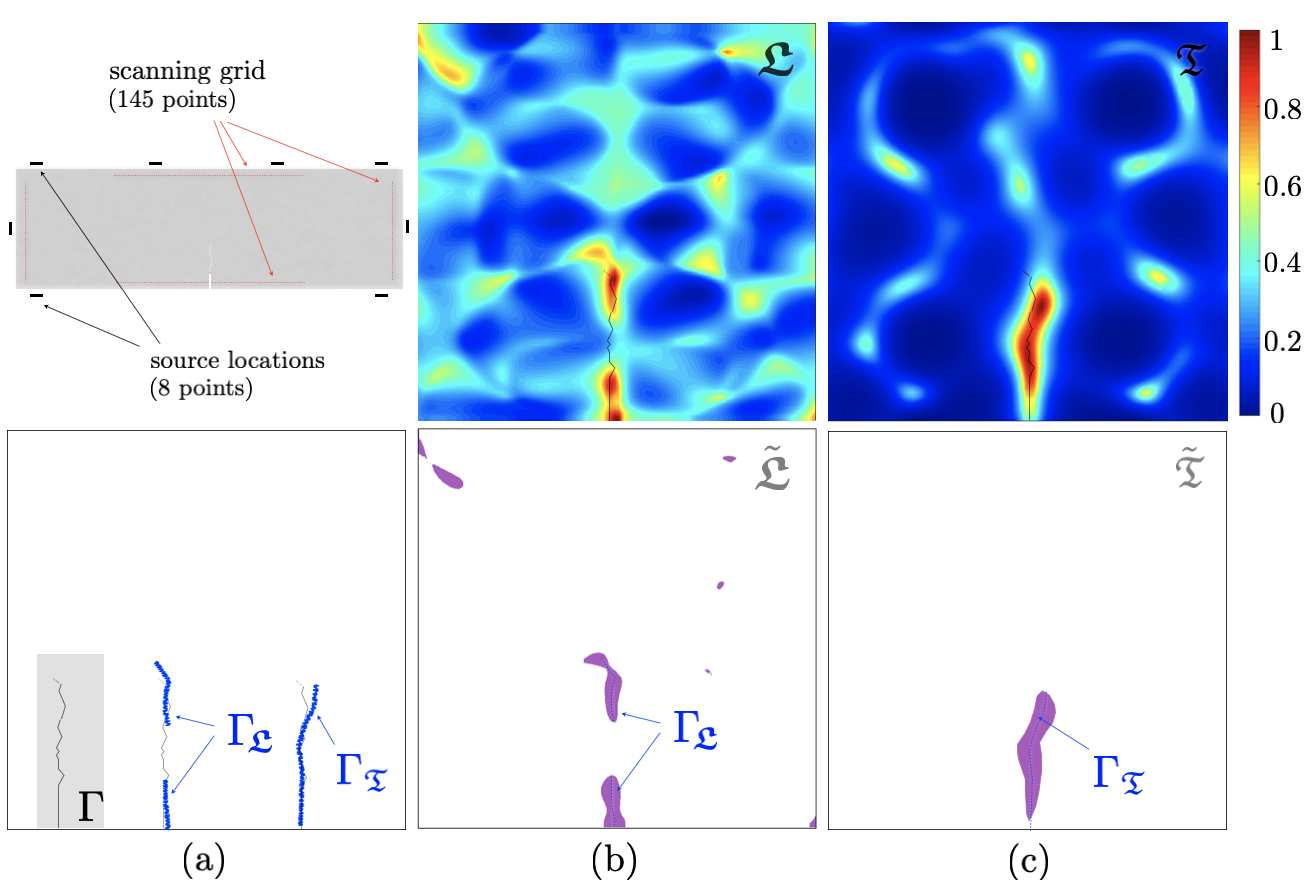} \vspace*{-2.5mm} 
\caption{\small{Time-~\emph{vs.}~frequency- domain reconstructions from post-fracturing data:~(a-top) sensing configuration, (a-bottom)~ground truth $\Gamma$ compared against the recovered ${\Gamma}_\mathfrak{T}$ and ${\Gamma}_{\mathfrak{L}}$ obtained from the thresholded maps,~(b)~multifrequency indicator map $\mathfrak{L}$ and its $60\%$ thresholded counterpart $\tilde{\mathfrak{L}}$~\cite{Yue2021}, and~(c)~time-domain indicator map $\mathfrak{T}$~\eqref{TLSM-I} and its affiliated $\tilde{\mathfrak{T}}$~\eqref{TiFM} thresholded at $60\%$.}} \lb{EIF}
\vspace*{0.5mm}
\end{figure} 

With reference to~\eqref{TiFM}, the thresholded maps $\tilde{\mathfrak{T}}$ and $\tilde{\mathfrak{L}}$ in Fig.~\ref{EIF} identify the support of sampling points $\bz$ where their associated imaging functional satisfies $\mathfrak{I}({\bz\nxs}) > 0.6 \nxs\times\nxs \text{max}(\mathfrak{I})$, $\mathfrak{I} \in \lbrace \mathfrak{T}, \mathfrak{L} \rbrace$. These maps are then used to approximate the fracture boundary $\Gamma_{\mathfrak{T}}$ and $\Gamma_{\mathfrak{L}}$ by drawing the mid-line through the thresholded damage zone as shown in Fig.~\ref{EIF} (bottom row). Comparing the recovered $\Gamma_{\mathfrak{T}}$ and $\Gamma_{\mathfrak{L}}$ with the ground truth $\Gamma$ further reveals the imaging ability of each indicator. 
 
\subsubsection{Sparse reconstruction}

To further investigate the performance of time-domain indicator with limited data, the scanning points on $\Gamma_m$ are uniformly downsampled to $N_m \in \lbrace 48, 28, 20, 16 \rbrace$ while the number of sources on $\Gamma_i$ remains $N_i =8$. The imaging functional $\mathfrak{T}$ is then recalculated using reduced data. The reconstruction results at $90\%$ and $75\%$ of the maximum load are shown in Fig.~\ref{DSTD}. The time- versus frequency- domain inversion results using reduced post-fracturing data are provided in Fig.~\ref{DSTF}. The TLSM indicator $\mathfrak{T}$ seem to remain robust with sparse data. 

\begin{figure}[!h]
\center\includegraphics[width=0.80\linewidth]{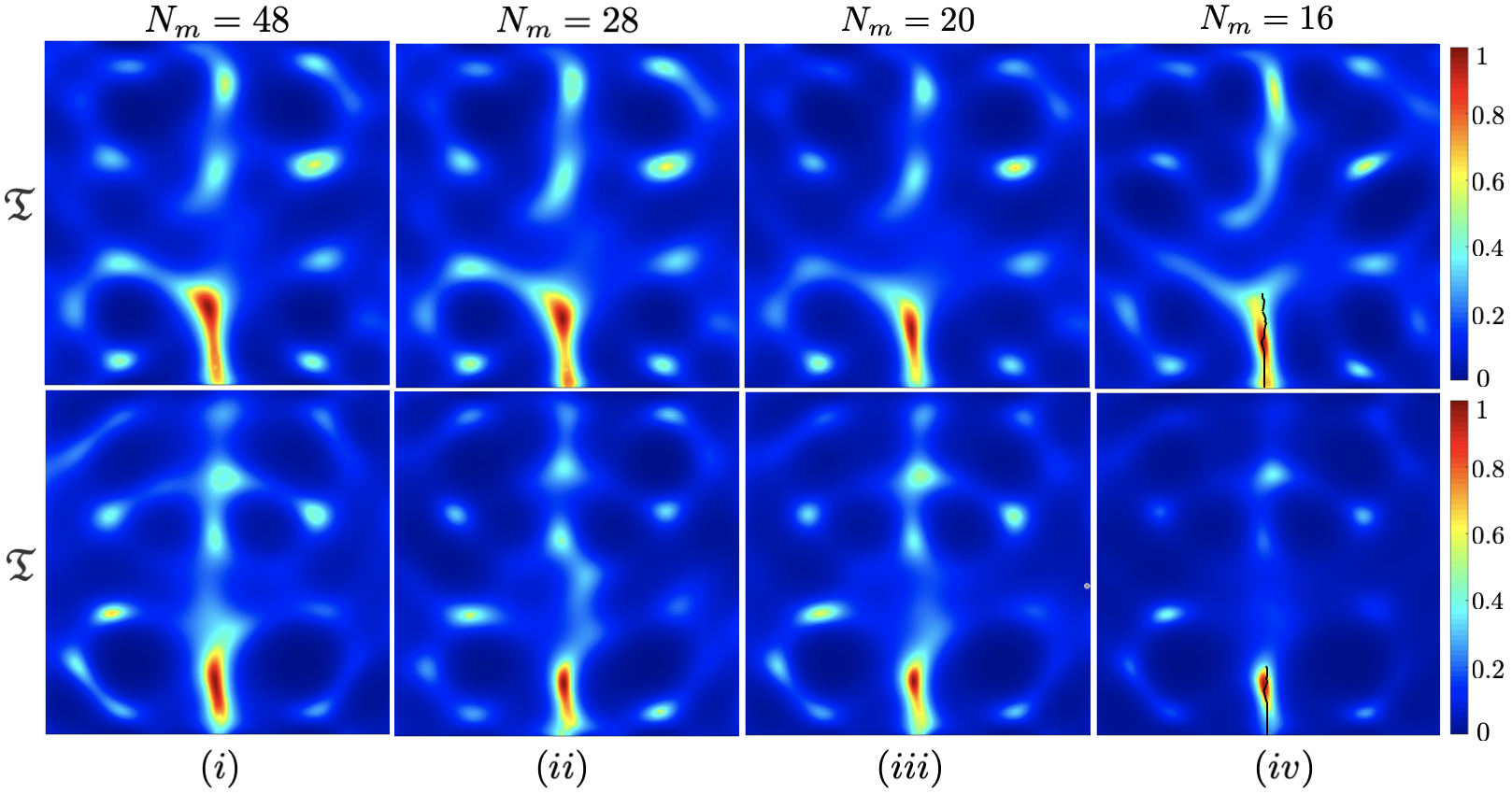} \vspace*{-3.5mm} 
\caption{\small{Time-domain reconstruction from reduced data at $75\%$ (top) and $90\%$ (bottom) of the maximum load where $\Gamma_m$ is sampled by $N_m$ points.}} \lb{DSTD}
\vspace*{-1.5mm}
\end{figure} 

\begin{figure}[!h]
\center\includegraphics[width=0.80\linewidth]{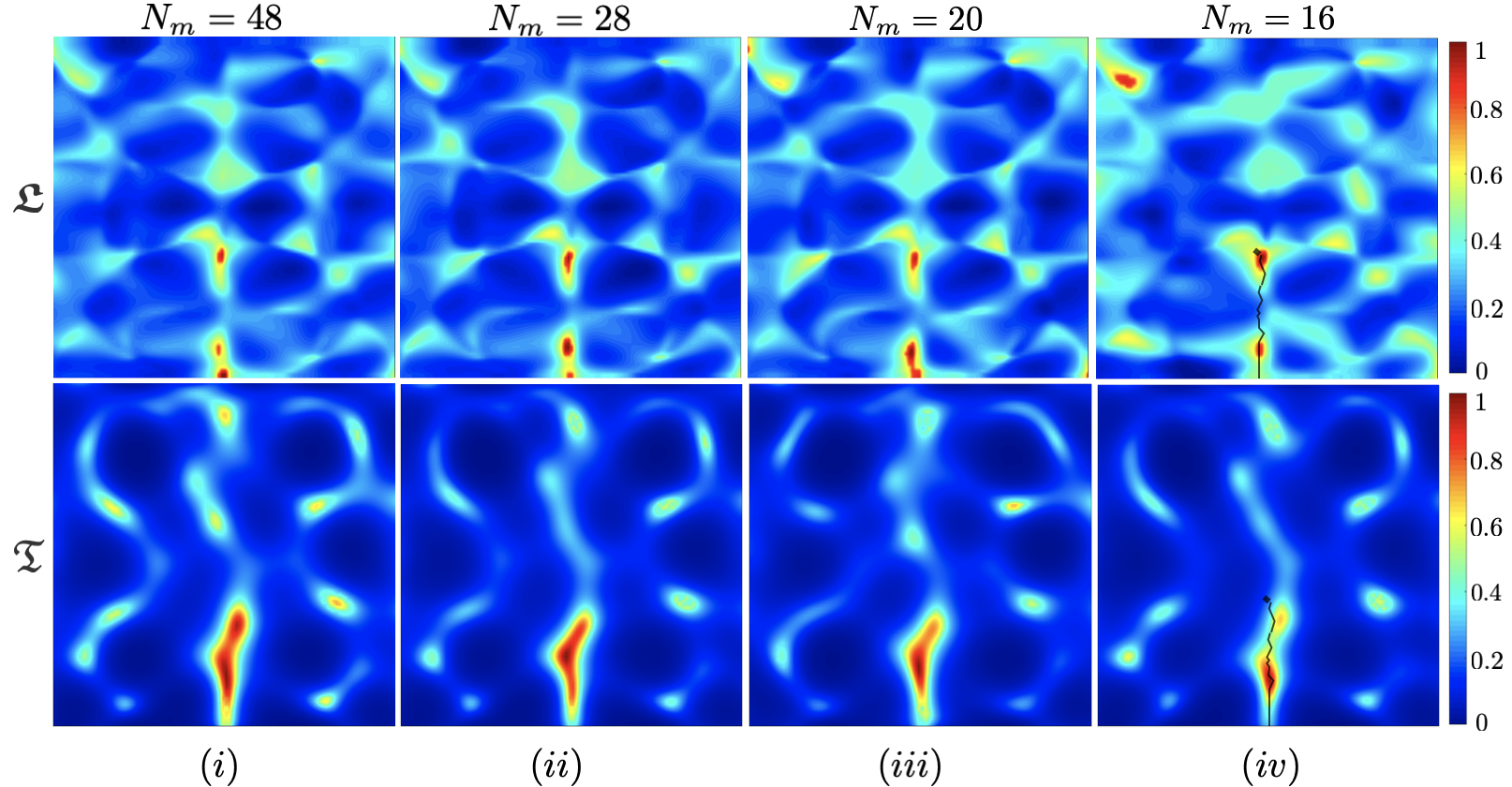} \vspace*{-2.5mm} 
\caption{\small{Time-~\emph{vs.}~frequency- domain inversion from reduced data at $60\%$ of the maximum load:~(top)~multifrequency indicator $\mathfrak{L}$~\cite{Yue2021}, and (bottom)~time-domain indicator $\mathfrak{T}$. Here, the number of scanning points on $\Gamma_m$ is $N_m$.}}\lb{DSTF}
\vspace*{-1.5mm}
\end{figure} 

\subsubsection{Reduced aperture reconstruction}

Partial-aperture and one-sided reconstructions are conducted in the time domain for sensing configurations shown in~Fig.~\ref{TEIF} where the results at $90\%$ and $75\%$ of the maximum load are illustrated. The comparison between $\mathfrak{T}$ and $\mathfrak{L}$ distributions using reduced-aperture data is provided in Fig.~\ref{TFIF}.  
 
\begin{figure}[!h]
\center\includegraphics[width=0.65\linewidth]{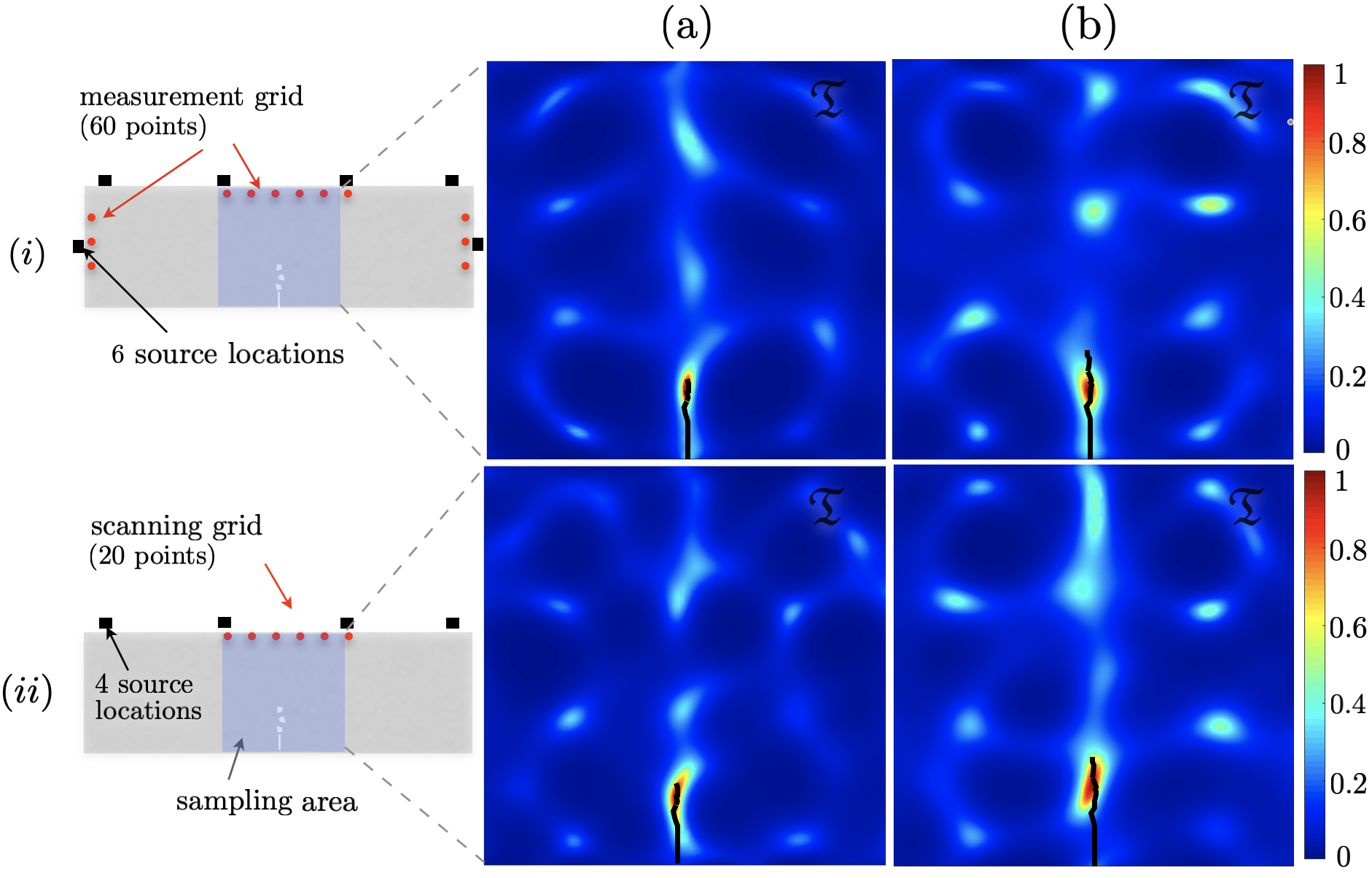} \vspace*{-2.5mm} 
\caption{\small{Partial-aperture time-domain reconstruction via~\eqref{TLSM-I} at~(a)~$90\%$ and~(b)~$75\%$ of the maximum load. The loci and number of source/measurement points (for each case) is indicated in the left column.}} \lb{TEIF}
\vspace*{-0.0mm}
\end{figure}

\begin{figure}[!h]
\center\includegraphics[width=0.65\linewidth]{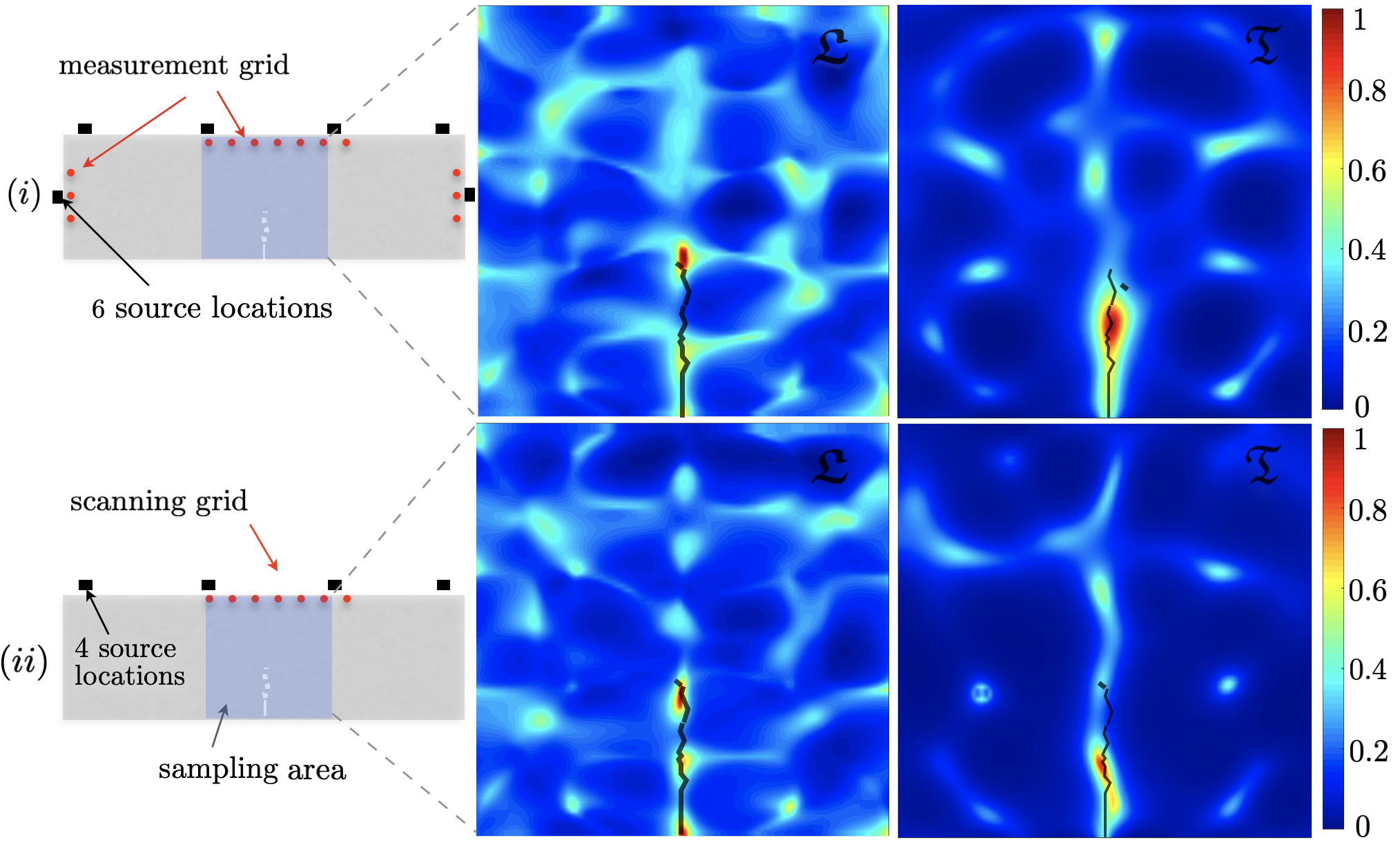} \vspace*{-2.5mm} 
\caption{\small{Time-~\emph{vs.}~frequency- domain inversion using (top) partial-aperture, and (bottom) one-sided data collected after fracturing. The sensing configuration for each row is indicated in the left column. The frequency domain $\mathfrak{L}$ maps are from~\cite{Yue2021}.}} \lb{TFIF}
\vspace*{-2.5mm}
\end{figure} 

\section{Conclusion}\lb{Conc}

This work provides the theoretical foundation of the time-domain linear sampling method for elastic-wave imaging of fractures which complements the (existing) LSM framework in the frequency domain, and thus, paves the way for a systematic comparison between time- and frequency- domain waveform inversion using laboratory experimental data of~\cite{Yue2021,pour2021}. The experiments reported by~\cite{Yue2021} (\emph{resp.}~\cite{pour2021}) feature interaction of ultrasonic waves with a stationary (\emph{resp.}~evolving) fracture in a plate whose signature on the specimen's boundary is captured for nondestructive evaluation. The TLSM indicator is applied to the scattered field data captured (a) at $90\%$ and $75\%$ of the maximum load in the post peak regime during propagation~\cite{pour2021}, and (b)~after the end of fracturing (occurred at $60\%$ of the maximum load)~\cite{Yue2021}. The TLSM maps affiliated with the sequential datasets in (a) and (b) successfully recover the spatiotemporal evolution of damage in the specimen. It is further shown that the reconstruction with sparse i.e.,~downsampled and/or reduced-aperture data remain robust at moderate noise levels. Using dataset (b), in parallel, a comparative analysis is conducted between the TLSM reconstructions and the corresponding multifrequency LSM maps reported by~\cite{Yue2021}. A remarkable contrast in image quality -- in terms of localization and presence of artifacts, is observed between the time- and frequency- domain inversions. The better quality of TLSM images are attributed to the full-waveform inversion in time (in addition to space) which involves both amplitude and phase information over the entire spectra during inversion.   

\section{Acknowledgements}\lb{Ac}   

XL is partly supported by the NSFC of China grant 12201023. FP and JS kindly acknowledge the support provided by the National Science Foundation (Grant No. 1944812) and the University of Colorado Boulder through FP's startup. The experimental data used in this work are taken from~\cite{Yue2021,pour2021}. This work utilized resources from the University of Colorado Boulder Research Computing Group, which is supported by the National Science Foundation (awards ACI-1532235 and ACI-1532236), the University of Colorado Boulder, and Colorado State University.

\bibliography{inverse,crackbib}

\end{document}